\documentclass{amsart}
\usepackage[normalem]{ulem}
\usepackage{amssymb,epsfig,mathrsfs,mathpazo,stackengine,scalerel,url,amsthm,thmtools,bbm}
\usepackage{xcolor}
\usepackage{afterpage}
\usepackage{stmaryrd}
\usepackage{epsfig,bm}
\usepackage[multiple]{footmisc}
\usepackage{accents} 
\usepackage{mathtools} 
\usepackage{footmisc}

\newcommand{\osc}{{\rm osc}}
\renewcommand{\Re}{\operatorname{Re}}

\usepackage{arydshln, booktabs, makecell, pbox, tabularx}

\usepackage{graphicx, caption, subcaption}

\usepackage{cleveref, algorithm, algpseudocode}

\newtheorem{theorem}{Theorem}[section]
\newtheorem{lemma}[theorem]{Lemma}
\newtheorem{proposition}[theorem]{Proposition}
\newtheorem{corollary}[theorem]{Corollary}

\declaretheorem[style=definition,qed=$\vartriangle$,sibling=theorem]{example}
\declaretheorem[style=remark,qed=$\vartriangle$,sibling=theorem]{remark}
\numberwithin{equation}{section}

\newcommand{\eps}{\varepsilon}

\newcommand{\R}{\mathbb R}

\newcommand{\N}{\mathbb N}

\newcommand{\cF}{\mathcal F}

\newcommand{\cE}{\mathcal E}

\newcommand{\cL}{\mathcal L}
\newcommand{\cS}{\mathcal S}
\newcommand{\cX}{\mathcal X}
\newcommand{\cY}{\mathcal Y}
\newcommand{\Lis}{\cL\mathrm{is}}
\newcommand{\cLis}{\cL\mathrm{is}_c}

\newcommand{\identity}{\mathrm{Id}}

\DeclareMathOperator{\ran}{ran}

\DeclareMathOperator{\diam}{diam}
\DeclareMathOperator*{\argmin}{argmin}
\DeclareMathOperator{\dist}{dist}

\DeclareMathOperator{\Span}{span}
\DeclareMathOperator{\curl}{curl}
\newcommand{\1}{\mathbb 1}

\makeatletter
\newcommand*\bigcdot{\mathpalette\bigcdot@{.5}}
\newcommand*\bigcdot@[2]{\mathbin{\vcenter{\hbox{\scalebox{#2}{$\m@th#1\bullet$}}}}}
\makeatother

\DeclareMathOperator{\divv}{div}

\newcommand{\nrm}{| \! | \! |}

\newcommand{\be}{\begin{equation}}
\newcommand{\ee}{\end{equation}}

\newcommand{\bbT}{\mathbb{T}}
\newcommand{\tria}{{\mathcal T}}
\newcommand{\RT}{\mathit{RT}}

{\par\begin{samepage}%
\begin{tabbing}\ttfamily%
 \hspace*{5mm}\=\hspace{3ex}\=\hspace{3ex}\=\hspace{3ex}\=\hspace{3ex}%
\=\hspace{3ex}\=\hspace{3ex}\=\hspace{3ex}\=\hspace{3ex}\kill}%
{\end{tabbing}\end{samepage}}

\newcounter{ccondition}

{%
\renewcommand{\theequation}{\temp}%
\addtocounter{equation}{-1}%
\par\noindent%
\ignorespacesafterend%
}

\newcounter{mylistcounter}
\renewcommand{\themylistcounter}{(\roman{mylistcounter})}

\newenvironment{mylist}{
\begin{list}{\themylistcounter.}{\usecounter{mylistcounter}
\setlength{\labelwidth}{-\smallskipamount}
\setlength{\labelsep}{\medskipamount}
\setlength{\topsep}{\smallskipamount}
\setlength{\itemsep}{\smallskipamount}
\setlength{\itemindent}{0cm}
\setlength{\leftmargin}{0cm}}}
{\end{list}}

\makeatletter
\@namedef{subjclassname@2020}{%
  \textup{2020} Mathematics Subject Classification}
\makeatother

\title[Quasi-Optimal Least Squares for PDEs]{Quasi-Optimal Least Squares: Inhomogeneous boundary conditions, and application with machine learning}
\date{\today}

\author{Harald Monsuur, Robin Smeets, Rob Stevenson}
\address{Korteweg-de Vries (KdV) Institute for Mathematics, University of Amsterdam, P.O. Box 94248, 1090 GE Amsterdam, The Netherlands.}
\email{h.monsuur@uva.nl, robinsmeets99@gmail.com, rob.p.stevenson@gmail.com}

\thanks{This research has been supported by the NSF Grant DMS ID 1720297, and by the Netherlands Organization for Scientific Research (NWO) under contract.~no.~SH-208-11. 
We acknowledge the support of SURF (www.surf.nl) in using the National Supercomputer Snellius.}



\begin{document}

\begin{abstract} We construct least squares formulations of PDEs with inhomogeneous essential boundary conditions, where boundary residuals are not measured in unpractical fractional Sobolev norms, but which formulations nevertheless are shown to yield a quasi-best approximations from the employed trial spaces. Dual norms do enter the least-squares functional, so that solving the least squares problem amounts to solving a saddle point or minimax problem. For finite element applications we construct uniformly stable finite element pairs, whereas for Machine Learning applications we employ adversarial networks.
\end{abstract}
\maketitle

\section{Introduction} \label{sec:1}
This paper is about Least Squares discretizations of boundary value problems (BVPs).  A comprehensive monograph on this topic is \cite{23.5}. In an abstract framework we consider variational formulations of BVPs in the form $Gu=f$, where for some Hilbert spaces $X$ and $Y$, $G$ is a linear operator $X \rightarrow Y'$ for which 
$\nrm \bigcdot\nrm_X:=\|G\bigcdot\|_{Y'}$ is a norm on $X$ that is equivalent to $\|\bigcdot\|_X$. In particular, $G$ is injective, but \emph{not} necessarily surjective.

Given a closed \emph{linear subspace} $X^\delta \subset X$, typically of finite element type, from an (infinite) family $\{X^\delta\}_{\delta \in \Delta}$ of such linear subspaces,
$$
u^\delta:=\argmin_{w \in X^\delta}\tfrac12\|f-Gw\|^2_{Y'}
$$
is the \emph{best} approximation to $u$ from $X^\delta$ w.r.t.~$\nrm\bigcdot\nrm_X$, and so a \emph{quasi-best} approximation w.r.t.~$\|\bigcdot\|_X$.

Not each BVP can be formulated in the above form with an evaluable norm \mbox{$\|\bigcdot\|_{Y'}$,} as when $Y' \simeq Y$ is an $L_2$-type space.
Such formulations are scarce in particular for the case of having essential inhomogeneous boundary conditions. The only exception we are aware of is \cite{20.187} for an inhomogeneous Robin boundary condition.
In general, $Y$ is a product of spaces, some with and others without evaluable dual norms.  Below we describe the approach to deal with the non-evaluable norms from our recent work \cite{204.19}.

\subsection{Approach from \cite{204.19}}
It suffices to consider the case that
$Y=Y_1 \times Y_2$, and so $G=(G_1,G_2)$ and $f=(f_1,f_2)$, where $\|\bigcdot\|_{Y_2'}$ can and $\|\bigcdot\|_{Y_1'}$ cannot be evaluated.
Then, for a sufficiently large closed \emph{linear subspace} $Y_1^\delta \subset Y_1$, typically of finite element type, in the Least Squares minimization over $X^\delta$ one replaces the norm $\|\bigcdot\|_{Y_1'}$ by the discretized dual-norm $\sup_{0 \neq v \in Y_1^\delta} \frac{|\bigcdot(v)|}{\|v\|_{Y_1}}$.
Assuming the pair $(X^\delta,Y_1^\delta)$ satisfies an (uniform) inf-sup or LBB condition, the resulting $u^\delta$ is still a \emph{quasi-best} approximation to $u$ from $X^\delta$.
This $u^\delta$ can be computed as the 2nd component of the pair $(\lambda^\delta,u^\delta) \in Y_1^\delta \times X^\delta$ that solves the saddle-point system
\begin{align*}
\langle \lambda^\delta,\mu\rangle_{Y_1}+(G_1 u^\delta)(\mu) &= f_1(\mu)\quad(\mu \in Y_1^\delta),\\
(G_1 w)(\lambda^\delta) -\langle G_2 u^\delta,G_2 w\rangle_{Y_2'}&=-\langle f_2,G_2w\rangle_{Y_2'} \quad (w \in X^\delta).
\end{align*}

A problem, however, arises when both $\|\bigcdot\|_{Y_1'}$ \emph{and} $\|\bigcdot\|_{Y_1}$ cannot be evaluated, as when $Y_1$ is a fractional Sobolev space. Such spaces naturally arise with the imposition of inhomogeneous essential boundary conditions. 
For $K^\delta_{Y_1}\colon {Y_1^\delta}' \rightarrow Y_1^\delta$ being such that $\sqrt{((K^\delta_{Y_1})^{-1} \bigcdot)(\bigcdot)}$ is uniformly equivalent to $\|\bigcdot|_{Y_1^\delta}\|_{Y_1}$, a solution is given by \emph{replacing} $\langle \bigcdot,\bigcdot\rangle_{Y_1}$ in the above saddle-point system by  $((K^\delta_{Y_1})^{-1}\bigcdot)(\bigcdot)$.
The resulting solution $u^\delta$ is then still \emph{quasi-best}, and by eliminating $\lambda^\delta$, it can be computed as the solution of the symmetric positive definite system 
$$
(G_1 u^\delta-f_1)(K_{Y_1}^\delta G_1 w)+\langle G_2 u^\delta-f_2,G_2 w\rangle_{Y_2'}=0 \quad(w \in X^\delta).
$$
We will refer to $K_{Y_1}^\delta$ as a \emph{preconditioner} for $Y_1^\delta \rightarrow (Y_1^\delta)'\colon \lambda^\delta \mapsto \langle \lambda^\delta,\bigcdot\rangle_{Y_1}$.
Such preconditioners, whose application, moreover, can be performed in linear time, are available for fractional Sobolev spaces of positive and negative order. It is, however, fair to say that their implementation is demanding. A related first work in this direction was \cite{249.06}.

\subsection{Current work}
Here we propose a different approach to deal with Sobolev spaces on the boundary with smoothness indices $\pm \frac12$.
It is based on the observation that for a domain $\Omega \subset \R^d$, and $\gamma^{\vec{n}}\colon H(\divv;\Omega) \rightarrow H^{-\frac12}(\partial\Omega)$ being the normal trace operator,
\begin{align} \label{eq:100}
&\|\bigcdot\|_{H^{\frac12}(\partial\Omega)} \eqsim \sup_{0 \neq \vec{v}\in H(\divv;\Omega)} \frac{|\int_{\partial\Omega} \bigcdot\, \gamma^{\vec{n}}\vec{v}  \,d s|}{\|\vec{v}\|_{H(\divv;\Omega)}}= \|{\gamma^{\vec{n}}}' \bigcdot\|_{H(\divv;\Omega)'}, 
\intertext{and analogously } \label{eq:101}
&\|\bigcdot\|_{H^{-\frac12}(\partial\Omega)} \eqsim  \|\gamma' \bigcdot\|_{H^1(\Omega)'},
\end{align}
 with $\gamma\colon H^1(\Omega) \rightarrow H^{\frac12}(\partial\Omega)$ denoting the standard trace operator.

\begin{remark}
With $H_\triangle(\Omega):=\{v \in H^1(\Omega)/\R\colon \triangle v \in L_2(\Omega)\}$ equipped with the graph norm, an alternative for \eqref{eq:100} is
\be \label{eq:102}
\|\bigcdot\|_{H^{\frac12}(\partial\Omega)}\eqsim\sup_{0 \neq \vec{v}\in H_\triangle(\Omega)} \frac{|\int_{\partial\Omega} \bigcdot \,\, \partial_{\vec{n}} v \,d s|}{\|v\|_{H_\triangle(\Omega)}}.
\ee
Although not attractive for finite element computations, the use of the smaller scalar space $H_\triangle(\Omega)$ instead of the vectorial $H(\divv;\Omega)$ shows to be beneficial in machine learning applications.
\end{remark}

By applying these alternative expressions \eqref{eq:100} (or \eqref{eq:102}) and \eqref{eq:101} for the Sobolev norms with smoothness indices $\frac12$ and $-\frac12$, for both 2nd order elliptic equations and the stationary Stokes equations on a domain $\Omega$, we obtain variational formulations of the form $G u=f$, with $\nrm\bigcdot\nrm_X:=\|G \bigcdot\|_{Y'}$ equivalent to $\|\bigcdot\|_X$, where both $X$ and $Y$ are products of \emph{merely functions spaces on $\Omega$}, which are either Sobolev spaces with smoothness indices in $\N_0$, or are equal to $H(\divv;\Omega)$. In the case of mixed boundary conditions some of these spaces are restricted by the incorporation of homogeneous boundary conditions on part of the boundary.
For either mixed or non-mixed boundary conditions, the operator $G$ will not be surjective, which for least squares discretizations does not hurt assuming $f \in \ran G$.

For finite element spaces $X^\delta \subset X$, and those factors $Y_i$ of the product space $Y$ that are not $L_2$-spaces, we will construct finite element spaces $Y_i^\delta \subset Y_i$ such that $(X^\delta,Y_i^\delta)$ satisfies the required LBB stability, so that $\|\bigcdot\|_{Y_i'}$ can be replaced by $\sup_{0 \neq v \in Y_i^\delta} \frac{|\bigcdot(v)|}{\|v\|_{Y_i}}$ whilst maintaining \emph{quasi-optimality} of the obtained Least Squares approximation $u^\delta \in X^\delta$. 

Compared to our earlier work, the advantage of this approach is that it does not require the application of preconditioners for fractional Sobolev norms on the boundary.

\subsection{Application with Machine Learning}
The approximation of the solution of a BVP using Neural Nets requires its formulation as a minimization problem.
For \emph{symmetric positive definite problems} a possibility is to use the \emph{Energy functional}, whereas a \emph{general applicable} approach is to use a \emph{Least Squares functional}. The imposition of essential boundary conditions, in particular inhomogeneous ones causes problems.

Let us illustrate this by considering the simple model problem
$$ \left\{
\begin{array}{r@{}c@{}ll}
-\triangle u&\,\,=\,\,& g &\text{ on } \Omega,\\
u &\,\,=\,\,& h &\text{ on } \partial\Omega,
\end{array}
\right.
$$
where with Least Squares methods non-symmetric first order terms can be added.
The solution $u$ of the above model problem is the minimizer over $\{w \in H^1(\Omega)\colon w=h \text{ on } \partial\Omega\}$ of the \emph{Energy functional}
$w \mapsto \frac12\int_\Omega |\nabla w|^2\,dx-g(w)$.

When $h=0$, one can approximate $u$ by the minimizer\footnote{For the purpose of the discussion, here we will assume that minima exist, and, moreover, that they can be computed.} $\widetilde{u}$ of this functional over
$\{\phi w\colon w \in \cX\}$, where $\cX$ is a \emph{set} of Neural Net functions, and $\phi$ is a, preferably smooth function, with $\phi(\bigcdot) \eqsim \dist(\bigcdot,\Omega)$. 
This $\widetilde{u}$ is the best approximation to $u$ from $\{\phi w\colon w \in \cX\}$ w.r.t.~$|\bigcdot|_{H^1(\Omega)}$.
For general domains $\Omega$, the construction of $\phi$ is, however, not obvious.
See \cite{255.5} for an extensive discussion.
Moreover, for non-smooth $u \in H^1_0(\Omega)$ the best approximation error from $\{\phi w\colon w \in \cX\}$ can be significantly larger than that from $\cX$ (cf. \cite[Thm.~7.6]{18.673}).

For $h \neq 0$, one can minimize the Energy functional over $\{\widetilde{h}+\phi w\colon w \in \cX\}$, where $\widetilde{h} \in H^1(\Omega)$ is an (approximate) extension of $h$ computed by transfinite interpolation (see \cite[Sect.~5.2]{255.5}), or by using a second neural net (\cite{248.55}). The appropriate norm for controlling $h-\widetilde{h}|_{\partial\Omega}$ is $\|\bigcdot\|_{H^{\frac12}(\partial\Omega)}$, which is, however, difficult to implement and therefore never used.

Instead of incorporating the essential boundary condition in the trial space, following \cite{70.25}, one may approximately enforce it by minimizing over $\cX$ the modified Energy functional $w \mapsto \frac12\int_\Omega |\nabla w|^2\,dx-g(w)+\alpha \|w-h\|_{L_2(\partial\Omega)}^2$, where $\alpha$ is some empirically chosen constant. The method is known as the \emph{Deep Ritz method}. Because of the use of the practically feasible $L_2(\partial\Omega)$-norm instead of the mathematically correct $H^{\frac12}(\partial\Omega)$-norm, generally it will not produce a quasi-best approximation to the solution from $\cX$.

Following \cite{239}, another way of imposing the essential boundary condition is by adding a penalty term to the Energy functional. Then for some $\alpha>0$ one minimizes
$$
w \mapsto \big[\int_\Omega  \tfrac12 |\nabla w|^2\,dx+\int_{\partial\Omega}  \tfrac12 \alpha w^2- w \partial_{\vec{n}} w
\,ds\big]-
\big[g(w) + \int_{\partial\Omega} \alpha  h w- h \partial_{\vec{n}} w\,ds\big]
$$
over $\cX$. This penalty method is called the \emph{Deep Nitsche method} (\cite{169.0556}).
Based on direct and inverse inequalities, for a finite element trial space an appropriate local choice for $\alpha$ is a sufficiently large constant times the reciprocal of the local mesh-size. For Neural Nets some empirically found global constant $\alpha$ is applied. In that case, however, it cannot be expected that  an appropriate, solution-independent choice for this penalty parameter exists.
\medskip

Existing \emph{Least Squares} based Neural Net approximations include minimization over $\cX$ of $w \mapsto \tfrac12 \|\triangle w +g\|_{L_2(\Omega)}^2+\alpha\|w-h\|_{L_2(\partial\Omega)}^2$ as with the \emph{Physics Informed Neural Network} (PINN) method \cite{243.65}.

The weak solution of Poisson's problem minimizes $\frac12 \|v \mapsto \int_\Omega \nabla w \cdot\nabla v\,dx-g(v)\|_{H^{-1}(\Omega)}$ over $\{w \in H^1(\Omega)\colon w|_{\partial\Omega}=h\}$. Invoking a second set of Neural Net functions $\cY$, this leads to the \emph{Weak Adversarial Network} (WAN) method \cite{19.28} of computing
$$
\argmin_{w \in \cX} \sup_{\{\phi v\colon v \in \cY\}} \tfrac12\frac{|\int_\Omega \nabla w \cdot\nabla (\phi v)\,dx -g(\phi v)|^2}{|\phi v|_{H^1(\Omega)}^2}+ \alpha\|w-h\|_{L_2(\partial\Omega)}^2.
$$

Because of the use of the $L_2(\partial\Omega)$-norm in which the boundary residual is measured, both PINN and WAN are not guaranteed to yield quasi-best approximations from the employed set $\cX$.

In \cite{35.9303} a first order system formulation is employed. Setting $\vec{p}=\nabla u$, the pair $(\vec{p},u)$ is the minimizer over $L_2(\Omega)^d \times \{w \in H^1(\Omega)\colon w=h \text{ on } \partial\Omega\}$ of $(\vec{q},w)\mapsto \tfrac12\|\vec{q}-\nabla w\|_{L_2(\Omega)}^2 + \tfrac12\|\divv \vec{q}+g\|_{L_2(\Omega)}^2$. For a set of Neural Net functions $\cX \subset L_2(\Omega)^d  \times H^1(\Omega)$, one approach would be to minimize this functional over $\{(\vec{q},\widetilde{h}+\phi w)\colon (\vec{q},w) \in \cX\}$ analogously to the method from \cite{248.55} discussed above, and with similar difficulties. Alternatively, one can minimize
$$
(\vec{q},w)\mapsto \tfrac12\|\vec{q}-\nabla w\|_{L_2(\Omega)}^2 + \tfrac12\|\divv \vec{q}+g\|_{L_2(\Omega)}^2+\tfrac12\|w-h\|_{H^{\frac12}(\partial\Omega)}^2
$$
over $\cX$, which would give a quasi-best approximation from $\cX$ to $(\vec{p},u)$ in the norm on $L_2(\Omega)^d  \times H^1(\Omega)$.
As noted in \cite{248.55}, however, for $d>1$ the computation of the fractional norm is unfeasible, and the numerical experiments are restricted to the one-dimensional case. In \cite{188}, in a slightly different context, it is proposed to approximately compute the Sobolev-Slobodeckij fractional norm,  which because of the singular kernel is difficult and in any case expensive.

A recent work on the use in Machine Learning of well-posed first order system formulations for homogeneous boundary conditions is \cite{241.4}.
\medskip

The \emph{Quasi-Optimal Least Squares} (QOLS) method that is introduced in this work solves the problem to correctly impose essential (inhomogeneous) boundary conditions. 
It measures the residuals of the PDE and the boundary conditions in correct norms, so that residual minimization is equivalent to error minimization, whereas it avoids the use of the non-evaluable fractional Sobolev norms. Notice that the approach of replacing such norms on test spaces by equivalent ones defined in terms of preconditioners is restricted to \emph{linear} test spaces, and so does not apply in the Machine Learning setting.

For the model problem in a first order formulation, given \emph{sets} of Neural Net functions $\cX \subset H(\divv;\Omega) \times H^1(\Omega)$ and $\cY \subset H_\triangle(\Omega)$, the QOLS method computes
\begin{align*}
\argmin_{(\vec{q},w) \in \cX } \Big[
\tfrac12 \|\vec{q} -\nabla w\|_{L_2(\Omega)^d}^2+\tfrac12\|\divv \vec{q}+&g\|_{L_2(\Omega)}^2+
\\
&\sup_{v \in \cY} \{\int_{\partial\Omega} (w-h)\partial_{\vec{n}}v\,ds-\tfrac12\|\vec{v}\|_{H_\triangle(\Omega)}^2\}\Big].
\end{align*}
The computed minimum will be shown to be a \emph{quasi-best} approximation to $(\vec{p},u)$ from $\cX$ w.r.t.~the norm on $H(\divv;\Omega) \times H^1(\Omega)$ under the condition that $\cY$ is sufficiently large in relation to $\cX$, akin to the LBB condition for linear trial- and test-spaces.

The QOLS method  can also be applied to the second order weak formulation. In that case, for $\cX \subset X:= H^1(\Omega)$, and
$\cY \subset Y:=H^1(\Omega) \times H_\triangle(\Omega)$, it computes
\begin{align*}
\argmin_{w\in \cX} \sup_{(v_1,v_2) \in \cY}\Big\{
\int_\Omega \nabla w \cdot\nabla (\phi v_1)\,dx - g(\phi v_1)+&\int_{\partial\Omega} (w-h)\partial_{\vec{n}}v_2\,ds\\
&- \tfrac12 \|\phi v_1\|_{H^1(\Omega)}^2-\tfrac12\|v_2\|_{H_\triangle(\Omega)}^2\Big\}.
\end{align*}
Again, for $\cY$ sufficiently large, only dependent on $\cX$, the computed minimum is a \emph{quasi-best} approximation to $u$ from $\cX$ in the norm on $H^1(\Omega)$. 
For larger $d$, this second order formulation has the advantage that no spaces of $d$-dimensional vector fields enter. On the other hand, one needs to construct 
a function $\phi$ with $\phi(\bigcdot) \eqsim \dist(\bigcdot,\Omega)$ to obtain test functions $\phi v_1 \in H_0^1(\Omega)$. Notice that a reduced approximation property as a consequence of the multiplication with $\phi$ is irrelevant at this test side.

\subsection{Organization} In Sect.~\ref{sec:2}, in an abstract setting we discuss least squares principles for the numerical approximation from \emph{linear subspaces} of the solution of a well-posed operator equation. In particular, we discuss the treatment of dual norms in a least squares functional. 

In Sect.~\ref{sec:3} we present well-posed formulations of the model elliptic second order boundary value problem on a domain $\Omega \subset \R^d$ with inhomogeneous boundary conditions.

The operators corresponding to the formulations from Sect.~\ref{sec:3} map into a product of Hilbert spaces, some of which being fractional Sobolev spaces on the boundary of $\Omega$.
In Sect.~\ref{sec:4}, well-posed modified formulations are constructed where all arising Hilbert spaces are integer order Sobolev spaces on $\Omega$, or duals of those.
Because of the presence of dual spaces, least-squares discretizations lead to saddle-point problems. The necessary uniform inf-sup stability conditions are verified for pairs of finite element spaces.

The least squares approach is not restricted to the model second order problem, and in Sect.~\ref{sec:5} + Appendix~\ref{sec:app} we consider the application to the stationary Stokes equations.

In Sect.~\ref{sec:6} we consider the discretisation of least squares formulations by \emph{(deep) neural nets}. Viewing a saddle-point problem as a minimax problem we employ adversarial networks. We derive a sufficient condition on the adversarial `dual' network for obtaining a quasi-best approximation from the `primal' network.

In Sect.~\ref{sec:7} we present numerical results for finite element and neural network discretizations, the latter in comparison with familiar neural net discretizations of PDEs.
Conclusions are formulated in Sect.~\ref{sec:conclusion}.

\subsection{Notations}
With the notation $C \lesssim D$, we mean that $C$ can be bounded by a multiple of $D$, independently of parameters which $C$ and $D$ may depend on, as the discretization index $\delta$.
Obviously, $C \gtrsim D$ is defined as $D \lesssim C$, and $C\eqsim D$ as $C\lesssim D$ and $C \gtrsim D$.

For Hilbert spaces $H$ and $K$, $\cL(H,K)$ will denote the space of bounded linear mappings $H \rightarrow K$ endowed with the operator norm $\|\bigcdot\|_{\cL(H,K)}$. The subset of invertible operators in $\cL(H,K)$, thus with inverses in $\cL(K,H)$,
 will be denoted as $\Lis(H,K)$. The subset of operators $R \in \Lis(H',H)$ for which $(R^{-1} \bigcdot)(\bigcdot)$ defines a scalar product on $H\times H$ will be denoted by $\cLis(H',H)$.

\section{Least squares principles in an abstract setting}  \label{sec:2}
\subsection{Continuous problem}
For some Hilbert spaces $X$ and $Y$, let $G\in \cL(X,Y')$ be a homeomorphism with its range, i.e.
\be \label{eq:1}
\|w\|_X \eqsim \|G w\|_{Y'} \quad (w \in X),
\ee
or, equivalently,
$$
M:=\sup_{0 \neq w \in X} \sup_{0 \neq v \in Y} \frac{|(G w)(v)|}{\|w\|_X \|v\|_{Y}} <\infty,\quad 
\alpha:=\inf_{0 \neq w \in X} \sup_{0 \neq v \in Y} \frac{|(G w)(v)|}{\|w\|_X \|v\|_{Y}} >0.
$$
Notice that $G$ is injective, but \emph{not} necessarily surjective.
Equipping $X$ with norm $\nrm \bigcdot\nrm_X:=\|G \bigcdot\|_{Y'}$, we have $\alpha \|\bigcdot\|_{X} \leq \nrm \bigcdot\nrm_X \leq M\|\bigcdot\|_{X}$.

Given $f \in Y'$, consider the Least Squares problem of finding
\be \label{eq:2}
u:=\argmin_{w \in X} \tfrac12\|Gw-f\|^2_{Y'}.
\ee
Necessarily, this $u$ solves the corresponding Euler-Lagrange equations
\be \label{eq:3}
\langle Gu,Gw\rangle_{Y'}=\langle f ,G w\rangle_{Y'}\quad (w \in X),
\ee
meaning that $G u$ is the $Y'$-orthogonal projection of $f$ onto $\ran G$, and therefore $\nrm u\nrm_X \leq \|f\|_{Y'}$.
Whenever $Gu=f$ has a solution, i.e., $f \in \ran G$, it is the unique solution of \eqref{eq:2}, and $\nrm u\nrm_X = \|f\|_{Y'}$.

\subsection{Discretization} \label{sec:discretization}
For $\delta$ from some (infinite) index set $\Delta$, let $\{0\} \subsetneq X^\delta \subsetneq X$ be a closed, e.g.~finite dimensional \emph{linear subspace}, typically of finite element type. Then 
\be \label{eq:ls}
u^\delta:=\argmin_{w \in X^\delta} \tfrac12\|Gw-f\|^2_{Y'}
\ee
is the \emph{best} approximation to $u$ from $X^\delta$ w.r.t.~$\nrm\bigcdot\nrm_X$.
This $u^\delta$ is computable when $\|\bigcdot\|_{Y'}$ (and $G$) is evaluable, as when $Y' \simeq Y$ is an $L_2$-space.

Variational formulations of PDEs in the operator form $G u=f$, where \eqref{eq:1} holds \emph{and} $\|\bigcdot\|_{Y'}$ is evaluable are relatively scarce, and in any case they are not available in the case of inhomogeneous essential boundary conditions. 

In general, $Y'$ is a product of Hilbert spaces some of them with, and others without evaluable norms. To analyze this situation, it suffices to consider the case that $Y'=Y'_1 \times Y'_2$,\footnote{The single-factor case gives no additional difficulties.}  and so $f=(f_1,f_2)$, and $G=(G_1,G_2)$, where
$$
\|\bigcdot\|_{Y_2'} \text{ is evaluable, and } \|\bigcdot\|_{Y_1'} \text{ is not}.
$$

Given a closed linear subspace $Y^\delta_1=Y^\delta_1(X^\delta) \subset Y_1$ that is sufficiently large such that
\be \label{eq:defgamma}
\gamma^\delta := \inf_{\{w \in X^\delta\colon G_1 w \neq 0\}}\frac{\sup_{0 \neq v_1 \in Y^\delta_1} \frac{|(G_1 w)(v_1)|}{\|v_1\|_{Y_1}}}{\|G_1 w\|_{Y_1'}} >0,
\ee
we \emph{replace} the non-computable Least Squares approximation \eqref{eq:ls} by%
\footnote{\label{foottie} Often we use that for $g_i \in Y_i'$,
$\sup_{0 \neq v \in Y} \frac{|g_1(v_1)+g_2(v_2)|^2}{\|v_1\|_{Y_1}^2+\|v_2\|_{Y_2}^2}=
\sup_{0 \neq v_1 \in Y_1} \frac{|g_1(v_1)|^2}{\|v_1\|_{Y_1}^2}\!+\!
\sup_{0 \neq v_2 \in Y_2} \frac{|g_2(v_2)|^2}{\|v_2\|_{Y_2}^2}$. For use later, we note that the equality holds also true for \emph{sets} $Y_1$ and $Y_2$ that are closed under scalar mutiplication.
Consequently, if $Y_1=Y_{11}\times\cdots\times Y_{1k}$, $f_1=(f_{11},\ldots,f_{1k})$, and 
$G_1=(G_{11},\ldots,G_{1k})$, for inf-sup constants $\gamma^\delta_{11},\ldots,\gamma^\delta_{1k}$ defined similarly as $\gamma^\delta$ in \eqref{eq:defgamma}, it holds that
$\gamma^\delta \geq \min(\gamma^\delta_{11},\ldots,\gamma^\delta_{1k})$.}
\be \label{eq:minresprac}
u^\delta:=\argmin_{w \in X^\delta} \tfrac12 \big(\sup_{0 \neq v \in Y_1^\delta} \frac{|(G_1 w-f_1)(v)|^2}{\|v\|_{Y_1}^2}+
\|G_2 w-f_2\|_{Y'_2}^2\big),
\ee
which, as we will see, \emph{is} computable. For the latter, we will assume that $\|\bigcdot|_{Y_1^\delta}\|_{Y_1}$, or a uniformly equivalent norm, is evaluable.

As is well-known, an inf-sup condition like \eqref{eq:defgamma} relates to existence of a `Fortin interpolator' $\Pi^\delta \in \cL(Y_1,Y_1^\delta)$. The following formulation of this relation does not require injectivity of $G_1$ which is not guaranteed. It is used that $G_1 \in \cL(X,Y_1')$, being a consequence of $G \in \cL(X,Y')$.

\begin{theorem}[{\cite[Prop.~5.1]{249.992}}] \label{thm:Fortin} Assuming that $G_1 X^\delta \neq \{0\}$ and $Y_1^\delta \neq \{0\}$, let
\be \label{fortin}
\Pi^\delta \in \cL(Y_1,Y_1^\delta) \text{ with } (G_1 X^\delta)\big((\identity - \Pi^\delta)Y_1\big)=0.
\ee
Then $\gamma^\delta \geq \|\Pi^\delta\|_{\cL(Y_1,Y_1)}^{-1}$.

Conversely, if $\gamma^\delta >0$, then there exists a $\Pi^\delta$ as in \eqref{fortin}, being even a projector onto $Y_1^\delta$, with
$\|\Pi^\delta\|^{-1}_{\cL(Y_1,Y_1)} = \gamma^\delta$.
\end{theorem}

For datum in the range of the operator $G$, next we show that if $\inf_{\delta \in \Delta} \gamma^\delta>0$, then $u^\delta$ is a \emph{quasi-best} approximation to $u$.

\begin{theorem} \label{thm:quasi-opt} If $f \in \ran G$, then the solution $u^\delta \in X^\delta$ of \eqref{eq:minresprac} satisfies
$$
\nrm u-u^\delta \nrm_{X} \leq \tfrac{1}{\gamma^\delta} \inf_{w \in X^\delta} \nrm u-w\nrm_{X}.
$$
\end{theorem}

\begin{proof}
Using Footnote~\ref{foottie}, one infers that

$$
\widetilde{\gamma}^\delta := \inf_{0 \neq w \in X^\delta}\frac{\sup_{0 \neq v \in Y^\delta_1 \times Y_2} \frac{|(G w)(v)|}{\|v\|_{Y}}}{\|G w\|_{Y'}}  \geq \min(\gamma^\delta,1)=\gamma^\delta.
$$

With $Y^\delta:=Y_1^\delta \times Y_2$, it holds that 
$$
u^\delta:=\argmin_{w \in X^\delta} \tfrac12 \sup_{0 \neq v \in Y^\delta} \frac{|(G w-f)(v)|^2}{\|v\|_{Y}^2}.
$$
With $R^\delta \in \Lis({Y^\delta}',Y^\delta)$ defined by $g(v)=\langle R^\delta g,v\rangle_Y$ ($g \in {Y^\delta}'$, $v \in Y^\delta$), we have
$$
\tfrac12 \sup_{0 \neq v \in Y^\delta} \frac{|(G w-f)(v)|^2}{\|v\|_{Y}^2}=
\tfrac12 \sup_{0 \neq v \in Y^\delta} \frac{|\langle R^\delta (G w-f),v\rangle_Y|^2}{\|v\|_{Y}^2}=\tfrac12 \|R^\delta(G w-f)\|_Y^2,
$$
so that $u^\delta$ is the solution in $X^\delta$ of 
$$
\langle R^\delta G u^\delta,R^\delta G w\rangle_{Y}=\langle R^\delta f,R^\delta G w\rangle_{Y}\quad (w \in X^\delta).
$$

From $f \in \ran G$, we have $f=Gu$, and we conclude that $P^\delta:=u \mapsto u^\delta$ is a projector onto $X^\delta$.
By substituting $w=u^\delta$, we infer that
$$
\|R^\delta G u^\delta\|_{Y}^2 = \langle R^\delta Gu,R^\delta G u^\delta\rangle_{Y}=(Gu)(R^\delta G u^\delta) \leq \nrm u \nrm_{X}\|R^\delta G u^\delta\|_{Y},
$$
and so
$$
\widetilde{\gamma}^\delta \nrm u^\delta \nrm_{X}=\widetilde{\gamma}^\delta \|G u^\delta\|_{Y'}\leq \sup_{0 \neq v \in Y^\delta}\frac{|(G u^\delta)(v)|}{\|v\|_Y}
= \|R^\delta G u^\delta\|_{Y} \leq \nrm u \nrm_{X}.
$$
From $X^\delta \neq \{0\}$, $X^\delta \neq X$, we have $\nrm \identity-P^\delta\nrm_{\cL(X,X)}=\nrm P^\delta\nrm_{\cL(X,X)}$ (\cite{169.5,315.7}), so that
$$
\nrm u-u^\delta \nrm_X \leq \nrm \identity-P^\delta\nrm_{\cL(X,X)} \inf_{w \in X^\delta} \nrm u-w \nrm_X \leq \tfrac{1}{\widetilde{\gamma}^\delta} \inf_{w \in X^\delta} \nrm u-w \nrm_X. \qedhere
$$
\end{proof}



\subsection{Implementation} \label{sec:implementation}
With the Riesz' lift $R_1^\delta \in \Lis({Y_1^\delta}',Y_1^\delta)$ defined by
\be \label{eq:13}
g(v)=\langle R_1^\delta g,v\rangle_{Y_1} \quad(g \in {Y_1^\delta}',\,v \in Y_1^\delta),
\ee
an equivalent expression for $u^\delta$ defined in \eqref{eq:minresprac} is
$$
u^\delta=\argmin_{w \in X^\delta} \tfrac12 \big(
\|R_1^\delta (G_1 w-f)\|_{Y_1}^2
+
\|G_2 w-f_2\|_{Y'_2}^2\big),
$$
and so it solves the corresponding Euler-Lagrange equations
\be \label{eq:5}
(G_1 u^\delta-f_1)(R_1^\delta G_1 w)+\langle G_2 u^\delta-f_2,G_2 w\rangle_{Y_2'}=0 \quad(w \in X^\delta).
\ee
One may verify that
$$
(\gamma^\delta)^2 \nrm w \nrm_{X}^2 \leq (G_1 w)(R_1^\delta G_1 w)+\langle G_2 w,G_2 w\rangle_{Y_2'} \leq \nrm w \nrm_{X}^2 \quad (w \in X^\delta).
$$


Introducing $\lambda^\delta:=R_1^\delta(f_1-G_1 u^\delta)$, $u^\delta$ is the 2nd component of the pair $(\lambda^\delta,u^\delta) \in Y_1^\delta \times X^\delta$ that solves the \emph{saddle-point problem}
\be \label{eq:4}
\begin{split}
\langle \lambda^\delta,\mu\rangle_{Y_1}+(G_1 u^\delta)(\mu) &= f_1(\mu)\quad(\mu \in Y_1^\delta),\\
(G_1 w)(\lambda^\delta) -\langle G_2 u^\delta,G_2 w\rangle_{Y_2'}&=-\langle f_2,G_2w\rangle_{Y_2'} \quad (w \in X^\delta).
\end{split}
\ee

To solve \eqref{eq:4} efficiently using an iterative method, one needs `uniform' preconditioners for the `upper-left block' and the Schur complement equation, which is \eqref{eq:5}.
So one needs
\be \label{eq:precond}
K_{Y_1}^\delta \in \cLis({Y_1^\delta}',Y_1^\delta),\quad K_{X}^\delta\in \cLis({X^\delta}',X^\delta),
\ee 
with \emph{uniformly bounded norms}, and  \emph{uniformly bounded norms of their inverses},
and whose applications can be efficiently computed, preferably \emph{in linear time}. 
The last requirement relates to the basis that is applied on $Y_1^\delta$ or $X^\delta$, since constructing a preconditioner amounts to constructing an approximate inverse of the stiffness matrix corresponding to $\langle\bigcdot,\bigcdot\rangle_{Y_1}$ or $\langle\bigcdot,\bigcdot\rangle_{X}$.
In our applications preconditioners that satisfy these requirements will be available.
\medskip

Having such $K_{Y_1}^\delta$, a most likely even more efficient strategy is to \emph{replace} in \eqref{eq:minresprac}, or, equivalently, in \eqref{eq:4}, the scalar product $\langle \bigcdot,\bigcdot\rangle_{Y_1}$ on $Y_1^\delta \times Y_1^\delta$  by $((K_{Y_1}^\delta)^{-1}\bigcdot)(\bigcdot)$, and to solve the resulting $u^\delta \in X^\delta$ from the Schur complement equation 
\be \label{eq:minresprac2}
(G_1 u^\delta-f_1)(K_{Y_1}^\delta G_1 w)+\langle G_2 u^\delta-f_2,G_2 w\rangle_{Y_2'}=0 \quad(w \in X^\delta).
\ee
What is more, when both $\|\bigcdot\|_{Y_1'}$ \emph{and} $\|\bigcdot\|_{Y_1}$ are not evaluable, as when $Y_1$ is a fractional Sobolev space, then this is the \emph{only strategy} that leads to a computable quasi-best least squares approximation.

\begin{theorem}[{\cite[Thm.~3.6]{204.19}}{\protect \footnote{Mind the other definitions of $M^\delta$ and $m^\delta$ in \cite{204.19}.}}] \label{thm:weereen}
If $f \in \ran G$, then the solution $u^\delta \in X^\delta$ of \eqref{eq:minresprac2} satisfies
$$
\nrm u - u^\delta \nrm_{X}\leq \frac{\max(1,M^\delta)}{\gamma^\delta \min(1,m^\delta)} \inf_{w \in X^\delta} \nrm  u - w \nrm_{X},
$$
where
$$
M^\delta:=\sup_{0 \neq \mu \in Y^\delta_1}\frac{\|\mu\|_{Y_1}}{((K_{Y_1}^\delta)^{-1} \mu)(\mu)^\frac12},\quad
m^\delta:=\inf_{0 \neq \mu \in Y^\delta_1}\frac{\|\mu\|_{Y_1}}{((K_{Y_1}^\delta)^{-1} \mu)(\mu)^\frac12}.
$$
\end{theorem}

\begin{remark} \label{rem:4} By comparing \eqref{eq:5} with \eqref{eq:minresprac2}, notice that for $Y_1^\delta$ equipped with scalar product $((K_{Y_1}^\delta)^{-1}\bigcdot)(\bigcdot)$, the
Riesz' lift ${Y_1^\delta}' \rightarrow Y_1^\delta$ is the operator $K_{Y_1}^\delta$.
\end{remark}

The \emph{symmetric positive definite system} \eqref{eq:minresprac2} can be efficiently solved using Preconditioned Conjugate Gradients with preconditioner $K_{X}^\delta$.

\subsection{A posteriori error estimation}
An obvious modification of \cite[Prop.~3.8]{204.19} that takes into account that in the current work $G$ is not necessarily surjective shows the following result.
\begin{proposition}[{\cite[Prop.~3.8 and Rem.~3.9]{204.19}}] \label{prop:apost}
Let $\Pi^\delta \in \cL(Y_1,Y_1^\delta)$ be a Fortin interpolator as in Theorem~\ref{thm:Fortin}. Then for $f \in \ran G$ and $w \in X^\delta$, the (squared) error estimator
$$
\cE^\delta(w,f):=(G_1 w-f_1)(K_{Y_1}^\delta (G_1 w-f_1)) +\|G_2 w-f_2\|_{Y_2'}^2
$$
satisfies
\begin{align*}
\min(1,\big(M^\delta)^{-2}\big) \cE^\delta(w,f) \leq \nrm u-w\nrm_X^2 \leq &\max\big(1,2 (m^\delta)^{-2}\big) \|\Pi^\delta\|_{\cL(Y_1,Y_1)}^2 \cE^\delta(w,f) \\
& +2 \|(\identity-{\Pi^{\delta}}') f_1\|_{Y_1'}^2,
\end{align*}
where
$$
\osc^\delta(f_1):=\|(\identity-{\Pi^{\delta}}') f_1\|_{Y_1'} \leq \|\Pi^\delta\|_{\cL(Y_1,Y_1)} \inf_{0 \neq z \in X^\delta} \nrm u - z\nrm_X.
$$
\end{proposition}

\begin{remark} \label{rem:2}
If $\sup_{\delta \in \Delta} \frac{\max(1,M^\delta)}{\min(1,m^\delta)}<\infty$ and $\inf_{\delta \in \Delta} \gamma^\delta>0$, then by taking $\Pi^\delta$ such $\|\Pi^\delta\|_{\cL(Y_1,Y_1)}=1/\gamma^\delta$, we conclude that
$\cE^\delta(w,f) + \osc^\delta(f_1)^2 \eqsim \nrm u-w\nrm_X^2$.
The oscillation term is, however, not computable.

As we will see, in our applications $(X^\delta,Y^\delta)$ can be chosen such that it allows for the construction of Fortin interpolators
that are both uniformly bounded, and for which, for sufficiently smooth $f_1$, 
data oscillation is of a \emph{higher order} than what, in view of the order of $X^\delta$, can be expected for the best approximation error\linebreak $\inf_{0 \neq z \in X^\delta} \nrm u - z\nrm_X$. In that case the estimator $\cE^\delta(w,f)$ is thus not only efficient, but in any case asymptotically also reliable.
\end{remark}

For $u^\delta$ being the solution of \eqref{eq:minresprac2}, one has  available $G_i u^\delta-f_i$ ($i =1,2$) as well as the preconditioned residual $K_{Y_1}^\delta (G_1 u^\delta-f_1)$, so that $\cE^\delta(u^\delta,f)$ can be computed efficiently.

\begin{remark} \label{rem:3}
Proposition~\ref{prop:apost} is also applicable with $K_{Y_1}^\delta$ reading as the Riesz' lift $R_1^\delta$ from \eqref{eq:13} (cf. Remark~\ref{rem:4}),  in which case $M^\delta=m^\delta=1$.
This case is relevant when $\langle \bigcdot,\bigcdot\rangle_{Y_1}$ is evaluable, and $u^\delta$ is computed by solving the saddle-point system \eqref{eq:4}. Then the quantity $(G_1 u^\delta-f_1)(R_1^\delta (G_1 u^\delta-f_1))$ is available as $\|\lambda^\delta\|_{Y_1}^2$, and so $\cE^\delta(u^\delta,f)=\|\lambda^\delta\|_{Y_1}^2 + \|G_2 u^\delta-f_2\|_{Y_2'}^2$. In applications this (squared) error estimator splits into a sum of local error indicators which suggests an adaptive refinement procedure.
\end{remark}

\section{Application}  \label{sec:3}
\subsection{Model elliptic second order boundary value problem} \label{sec:model}
On a bounded Lipschitz domain $\Omega \subset \R^d$, where $d \geq 2$, and closed $\Gamma_D, \Gamma_N \subset \partial\Omega$, with $\Gamma_D \cup \Gamma_N =\partial\Omega$ and $|\Gamma_D \cap \Gamma_N|=0$, consider the following elliptic second order boundary value problem
\begin{equation} \label{bvp}
 \left\{
\begin{array}{r@{}c@{}ll}
-\divv A \nabla u+ B u &\,\,=\,\,& g &\text{ on } \Omega,\\
u &\,\,=\,\,& h_D &\text{ on } \Gamma_D,\\
\vec{n}\cdot A \nabla u &\,\,=\,\,& h_N &\text{ on } \Gamma_N,
\end{array}
\right.
\end{equation}
 where $B \in \cL(H^1(\Omega),L_2(\Omega))$, and $A(\bigcdot)=A(\bigcdot)^\top \in L_\infty(\Omega)^{d\times d}$ satisfies $\xi^\top A(\bigcdot) \xi \eqsim \|\xi\|^2$ ($\xi \in \R^d$).
We assume that the matrix $A$, and the first order operator $B$ are such that\footnote{When $\Gamma_D=\emptyset$, it can be needed to replace $H^1_{0,\Gamma_D}(\Omega)=H^1(\Omega)$ by $H^1(\Omega)/\R$. For simplicity, we do not consider this situation.}
$$
w\mapsto (v \mapsto \int_\Omega A \nabla w \cdot \nabla v +B w \, v \,dx)
\in \Lis\big(H^1_{0,\Gamma_D}(\Omega),H^1_{0,\Gamma_D}(\Omega)'\big).
$$

\subsection{Well-posed first order system reformulations} \label{sec:wellposed}
We consider two consistent first order system formulations \ref{second}--\ref{third} of \eqref{bvp}.\footnote{A third option is the first order ultra-weak formulation. In that formulation both Dirichlet and Neumann are natural and therefore do not require special attention. In Sect.~\ref{sec:applmodel} we will also consider the standard second order formulation.} 
From the formulations given, one easily derives the expressions for the operator `$G$', solution `$u$', right-hand side `$f$', and spaces `$X$' and `$Y$'. Implicitly we will assume that the data $g$, $h_D$, and $h_N$ are such that `$f$' is in dual of `$Y$'.

\begin{mylist}
\item({\bf mild formulation}) \label{second}
Find $(\vec{p},u) \in H(\divv;\Omega) \times H^1(\Omega)$ such that
$$
 \left\{
\begin{array}{r@{}c@{}ll}
\vec{p}-A \nabla u &\,\,=\,\,& 0 & \text{ in } L_2(\Omega)^d,\\
B u-\divv \vec{p}&\,\,=\,\,& g & \text{ in } L_2(\Omega),\\
 \gamma_{\Gamma_D}(u)&\,\,=\,\,& h_D & \text{ in } H^{\frac12}(\Gamma_D),\\
\gamma_{\Gamma_N}^{\vec{n}}(\vec{p})&\,\,=\,\,& h_N & \text{ in } H^{-\frac12}(\Gamma_N),
\end{array}
\right.
$$
where $ \gamma_{\Gamma_D}$ and $\gamma_{\Gamma_N}^{\vec{n}}$ are the trace or normal trace operators on $\Gamma_D$ and $\Gamma_N$, respectively, $H^{\frac12}(\Gamma_D)$ is the interpolation space $[L_2(\Gamma_D),H^1(\Gamma_D)]_{\frac12,2}$, 
$H^{-\frac12}(\Gamma_N):=H_{00}^{\frac12}(\Gamma_N)'$, where $H_{00}^{\frac12}(\Gamma_N):=[L_2(\Gamma_N),H^1_0(\Gamma_N)]_{\frac12,2}$.
The dual of $H^{\frac12}(\Gamma_D)$ is denoted by $\widetilde{H}^{-\frac12}(\Gamma_D)$.

\item({\bf mild-weak formulation}) \label{third}
Find $(\vec{p},u) \in L_2(\Omega)^d \times H^1(\Omega)$ such that
$$
 \left\{
\begin{array}{r@{}c@{}ll}
\vec{p}-A \nabla u &\,\,=\,\,& 0 & \text{ in } L_2(\Omega)^d,\\
\int_\Omega \vec{p}\cdot \nabla v+B u\,v\,dx &\,\,=\,\,& g(v) +\int_{\Gamma_N} h_N \gamma_{\Gamma_N}(v) \,ds& (v \in H^1_{0,\Gamma_D}(\Omega)),\\
 \gamma_{\Gamma_D}(u)&\,\,=\,\,& h_D & \text{ in } H^{\frac12}(\Gamma_D).
 \end{array}
\right.
$$
\end{mylist}

In \cite{249.96, 204.19} both formulations have been shown to be \emph{well-posed} in the strong sense that `$G$' is \emph{boundedly invertible} between `$X$' and the dual of `$Y$', which is \eqref{eq:1} together with surjectivity of `$G$'.

\begin{remark}Although not the focus of this work, we mention that for \emph{homogeneous} (essential) boundary conditions simplified formulations can be applied.
In \ref{second}, replace $H(\divv;\Omega) \times H^1(\Omega)$ by
$H_{0,\Gamma_N}(\divv;\Omega) \times H_{0,\Gamma_D}^1(\Omega)$ and omit the last two equations; and in \ref{third}, replace $H^1(\Omega)$ by $H_{0,\Gamma_D}^1(\Omega)$ and omit the last equation.
\end{remark}

\begin{remark} \label{rem:1} Formulation~\ref{second} has the advantage that both `field residuals' are measured in $L_2$-spaces. A disadvantage is that it requires that $g \in L_2(\Omega)$, whereas $g \in H^1_{0,\Gamma_D}(\Omega)'$ is allowed in \ref{third}.
Datum $g \in H^1_{0,\Gamma_D}(\Omega)'$ is, however,  covered when it is given as
\be \label{eq:6}
g(v)=\int_\Omega g_1 v-\vec{g}_2\cdot\nabla v\,dx, 
\ee
for some $g_1 \in L_2(\Omega)$ and $\vec{g}_2 \in L_2(\Omega)^d$. In that case, replace the two equations
$\vec{p}-A \nabla u = 0$ and $B u-\divv \vec{p}= g$ by $\vec{p}-A \nabla u = \vec{g}_2$ and $B u-\divv \vec{p}= g_1$.

Any $g\in H^1_{0,\Gamma_D}(\Omega)'$ can be written in the form \eqref{eq:6}, but in general it requires solving a PDE to find such a decomposition. In a finite element setting, an alternative approach is to replace  $g\in H^1_{0,\Gamma_D}(\Omega)'$ in \ref{second} by a suitable projection onto the space of piecewise polynomials. It was shown that the then resulting solution $(\vec{p}^\delta,u^\delta)$ is still quasi-optimal
(see \cite{75.068} for the lowest order case, and \cite[Rem.~4.7]{204.19} for the extension to general orders).
\end{remark}

The key to derive well-posedness of \ref{second}--\ref{third} in the aforementioned strong sense was the following abstract lemma.
\begin{lemma}[{\cite[Lemma 2.7]{75.28}}] \label{lem:bi}
Let $Z$ and $V_2$ be Banach spaces, and let $V_1$ be a normed linear space.
Let $T \in \cL(Z,V_2)$ be surjective, and let $H \in \cL(Z,V_1)$ be such that 
with $Z_0:=\ker T$, $H \in \Lis(Z_0,V_1)$. Then $(H,T) \in \Lis(Z,V_1 \times V_2)$.\footnote{If only $\|H\bigcdot\|_{V_1} \eqsim \|\bigcdot\|_Z$ on $Z_0$, then  $\|H\bigcdot\|_{V_1}+\|T\bigcdot\|_{V_2} \eqsim \|\bigcdot\|_Z$ on $Z$.}
\end{lemma}

This lemma shows that it suffices to find a surjective trace operator that corresponds to the essential boundary conditions, and to show well-posedness of the problem with homogeneous essential boundary conditions. 
\medskip

In \ref{second} or \ref{third}, the space $Y'$ reads as $L_2(\Omega)^d \times L_2(\Omega) \times H^{\frac12}(\Gamma_D) \times H^{-\frac12}(\Gamma_N)$ or as $L_2(\Omega)^d \times H_{0,\Gamma_D}^1(\Omega)' \times H^{\frac12}(\Gamma_D)$.
To apply to \ref{second}--\ref{third} the Least Squares discretisation \eqref{eq:minresprac}/\eqref{eq:4}, or the modified one \eqref{eq:minresprac2}, 
for those factors $Y'_i$ in $Y'=\prod_i Y'_i$ that are unequal to $L_2$-spaces, one has to select test spaces $Y^\delta_i \subset Y_i$
for which $\sup_{0 \neq v \in Y_i^\delta} \frac{|(G_i w)(v)|}{\|v\|_{Y_i}} \gtrsim \|G_i w\|_{Y'_i}$ ($w \in X^\delta$), being the (uniform) inf-sup stability condition \eqref{eq:defgamma}. 

Furthermore, (uniform) preconditioners in $\cLis({X^\delta}',X^\delta)$, and for aforementioned $Y_i$, in $\cLis({Y_i^\delta}',Y_i^\delta)$ have to be found, preferably of linear computational complexity. If $Y_i$ is a fractional Sobolev space, then having such a (uniform) preconditioner is even \emph{indispensible} to circumvent the evaluation of the fractional norm (see the paragraph preceding Theorem~\ref{thm:weereen}).

For both formulations~\ref{second}--\ref{third}, in \cite[Sect.~4]{204.19} trial- and test-spaces of finite element type that give (uniform) inf-sup stability have been constructed, and suitable preconditioners are known. 

The construction of test-spaces on the boundary, and the implementation of preconditioners for fractional Sobolev spaces on the boundary require quite some effort.
Therefore, in the following subsection we present modified formulations, that are well-posed in the sense of \eqref{eq:1}, in which all function spaces on the boundary, specifically fractional Sobolev spaces, disappeared.
Although the operators `$G$' will not be surjective anymore, recall that for consistent data Least Squares discretisations yield quasi-optimal approximations.

\section{Avoidance of fractional Sobolev norms} \label{sec:4}
\subsection{Modified first order formulations in terms of field variables only} \label{sec:field}
With the trace operator $\gamma_{\Gamma_N}:= v \mapsto v|_{\Gamma_N} $, it is known that
$$
\|\bigcdot\|_{H_{00}^{\frac12}(\Gamma_N)} \eqsim \inf\{\|v\|_{H^1(\Omega)}\colon v \in H^1_{0,\Gamma_D}(\Omega),\,\gamma_{\Gamma_N}(v)=\bigcdot\}.
$$
From $H^{-\frac12}(\Gamma_N)=H_{00}^{\frac12}(\Gamma_N)'$, it follows that on $H^{-\frac12}(\Gamma_N)$,
\be \label{eq:16}
\|\bigcdot\|_{H^{-\frac12}(\Gamma_N)}
\eqsim
\sup_{0 \neq v \in H^1_{0,\Gamma_D}(\Omega)} \frac{|\int_{\Gamma_N}  \bigcdot \,\gamma_{\Gamma_N} v\,ds|}{\|v\|_{H^1(\Omega)}}=
\|\gamma_{\Gamma_N}' \bigcdot\|_{H^1_{0,\Gamma_D}(\Omega)'}.
\ee

Similarly, for the normal trace operator $\gamma^{\vec{n}}_{\Gamma_D}:=
\vec{v} \mapsto \vec{v}|_{\Gamma_D} \bigcdot \vec{n}$, it holds that
$$
\|\bigcdot\|_{\widetilde{H}^{-\frac12}(\Gamma_D)} \eqsim \inf\{\|\vec{v}\|_{H(\divv;\Omega)}\colon \vec{v} \in H_{0,\Gamma_N}(\divv;\Omega),\,\gamma^{\vec{n}}_{\Gamma_D}(\vec{v})=\bigcdot\},
$$
see e.g.~\cite[Remark~3.8]{35.856} (there with $\widetilde{H}^{-\frac12}(\Gamma_D)$ denoted by $H^{-\frac12}(\Gamma_D)$).
Recalling that $H^{\frac12}(\Gamma_D)'=\widetilde{H}^{-\frac12}(\Gamma_D)$,
it follows that that on $H^{\frac12}(\Gamma_D)$,
\be \label{eq:15}
\|\bigcdot\|_{H^{\frac{1}{2}}(\Gamma_D)}\eqsim
\sup_{0 \neq \vec{v} \in H_{0,\Gamma_N}(\divv;\Omega)} 
\frac{|\int_{\Gamma_D}\bigcdot\, \gamma^{\vec{n}}_{\Gamma_D} \vec{v} \,ds|}{\|\vec{v}\|_{H(\divv;\Omega)}}=
\|\gamma^{\vec{n}\,{\textstyle '}}_{\Gamma_D}\bigcdot\|_{H_{0,\Gamma_N}(\divv;\Omega)'}.
\ee

From the well-posedness of \ref{second}, and \eqref{eq:16} and \eqref{eq:15}, we conclude that the following modified formulation \ref{secondaccent} is well-posed in the sense that the corresponding operator $G\in \cL(X,Y')$ is a homeomorphism with its range, i.e., \eqref{eq:1} is valid.




\renewcommand{\themylistcounter}{(\roman{mylistcounter})'}
\begin{mylist}
\item({\bf modified mild formulation}) \label{secondaccent}
Find $(\vec{p},u) \in H(\divv;\Omega) \times H^1(\Omega)$ such that
$$
 \left\{
\begin{array}{r@{}c@{}ll}
\vec{p}-A \nabla u &\,\,=\,\,& 0 & \text{ in } L_2(\Omega)^d,\\
B u-\divv \vec{p}&\,\,=\,\,& g & \text{ in } L_2(\Omega),\\
\gamma^{\vec{n}\,{\textstyle '}}_{\Gamma_D} \gamma_{\Gamma_D}(u)&\,\,=\,\,& \gamma^{\vec{n}\,{\textstyle '}}_{\Gamma_D}h_D & \text{ in } H_{0,\Gamma_N}(\divv;\Omega)',\\
\gamma_{\Gamma_N}' \gamma_{\Gamma_N}^{\vec{n}}(\vec{p})&\,\,=\,\,&\gamma_{\Gamma_N}'h_N & \text{ in } H^1_{0,\Gamma_D}(\Omega)'.
\end{array}
\right.
$$
Knowing that the operator corresponding to \ref{second} is surjective, the range of the current operator is $L_2(\Omega)^d \times L_2(\Omega) \times \gamma^{\vec{n}\,{\textstyle '}}_{\Gamma_D} H^{\frac12}(\Gamma_D) \times \gamma_{\Gamma_N}' H^{-\frac12}(\Gamma_N)$.

A subset of the above arguments applied to the formulation \ref{third} shows that following modified formulation \ref{thirdaccent} is well-posed in the sense that the corresponding operator is a homeomorphism with its range, i.e., satisfies \eqref{eq:1}.

\item({\bf modified mild-weak formulation}) \label{thirdaccent}
Find $(\vec{p},u) \in L_2(\Omega)^d \times H^1(\Omega)$ such that
$$
 \left\{
\begin{array}{r@{}c@{}ll}
\vec{p}-A \nabla u &\,\,=\,\,& 0 & \text{ in } L_2(\Omega)^d,\\
\int_\Omega \vec{p}\cdot \nabla v+B u\,v\,dx &\,\,=\,\,& g(v) +\int_{\Gamma_N} h_N \gamma_{\Gamma_N}(v) \,ds& (v \in H^1_{0,\Gamma_D}(\Omega)),\\
\gamma^{\vec{n}\,{\textstyle '}}_{\Gamma_D} \gamma_{\Gamma_D}(u)&\,\,=\,\,& \gamma^{\vec{n}\,{\textstyle '}}_{\Gamma_D}h_D & \text{ in } H_{0,\Gamma_N}(\divv;\Omega)'.
\end{array}
\right.
$$
The range of the corresponding operator is $L_2(\Omega)^d \times H^1_{0,\Gamma_D}(\Omega)' \times \gamma^{\vec{n}\,{\textstyle '}}_{\Gamma_D} H^{\frac12}(\Gamma_D)$.
\end{mylist}

Although not very relevant for finite element computations, for the application with machine learning, in particular when $\Omega$ is a higher dimensional space, we give the following alternative for \eqref{eq:15}. It avoids the introduction of the space $H(\divv;\Omega)$ of $d$-dimensional vector fields, at the expense of requiring smoother test functions.

\begin{lemma} \label{lem:extra} For the Hilbert space $H_\triangle(\Omega):=\{v \in H^1(\Omega)/\R\colon \triangle v \in L_2(\Omega)\}$ equipped with squared norm $\|v\|_{H_\triangle(\Omega)}^2:=\|\triangle v\|_{L_2(\Omega)}^2+|v|^2_{H^1(\Omega)}$, and its closed linear subspace $H_{\triangle,0,\Gamma_N}(\Omega):=\{v \in H_\triangle(\Omega)\colon \gamma_{\Gamma_N}^{\vec{n}}(\nabla v)=0\}$, it holds that on $H^{\frac{1}{2}}(\Gamma_D)$,
$$
\|\bigcdot\|_{H^{\frac{1}{2}}(\Gamma_D)}\eqsim
\sup_{0 \neq v \in H_{\triangle,0,\Gamma_N}(\Omega)} 
\frac{|\int_{\Gamma_D} \bigcdot \,\,\gamma_{\Gamma_D}^{\vec{n}}(\nabla v) \,ds|}{|v|_{H_\triangle(\Omega)}}=
\|(\gamma_{\Gamma_D}^{\vec{n}} \circ \nabla)'\bigcdot\|_{H_{\triangle,0,\Gamma_N}(\Omega)'}.
$$
\end{lemma}

\begin{proof} It holds that $\|\bigcdot\|_{H^\frac12(\Gamma_D)} \eqsim \inf\{\|z\|_{H^1(\Omega)}\colon z\in H^1(\Omega),\,\gamma_D(z)=\bigcdot\}$. Given $f \in \widetilde{H}^{-\frac12}(\Gamma_D)=H^{\frac12}(\Gamma_D)'$, define $v \in H^1(\Omega)$ by
\be \label{101}
\int_\Omega \nabla v\cdot\nabla z+v z \,dx=f(\gamma_D z)\quad (z \in H^1(\Omega)).
\ee
Then 
\begin{align*}
\|f\|_{\widetilde{H}^{-\frac12}(\Gamma_D)} &\eqsim \sup_{\{z \in H^1(\Omega)\colon \gamma_D(z)\neq 0\}}\frac{|f(\gamma_D z)|}{\inf\{\|\tilde{z}\|_{H^1(\Omega)}\colon \tilde{z}\in H^1(\Omega),\,\gamma_D(\tilde{z}-z)=0\}}\\
& =\sup_{\{z \in H^1(\Omega)\colon \gamma_D(z)\neq 0\}}\frac{|f(\gamma_D z)|}{\|z\|_{H^1(\Omega)}}=\|v\|_{H^1(\Omega)}.
\end{align*}
Since $f(\gamma_D \bigcdot)$ vanishes on $H^1_0(\Omega)$, \eqref{101} shows that $\triangle v=v \in L_2(\Omega)$, so that
$\|v\|_{H_\triangle(\Omega)}=\|v\|_{H^1(\Omega)} \eqsim \|f\|_{\widetilde{H}^{-\frac12}(\Gamma_D)}$.
Furthermore,
$$
\int_{\partial \Omega} z \nabla v \cdot {\vec{n}}\,ds=\int_\Omega \nabla v\cdot\nabla z+z \triangle v\,dx=f(\gamma_D z) \quad (z \in H^1(\Omega)),
$$
 or $\gamma_{\Gamma_D}^{\vec{n}}(\nabla v)=f$ and $\gamma_{\Gamma_N}^{\vec{n}}(\nabla v)=0$, and so $v \in H_{\triangle,0,\Gamma_N}(\Omega)$. We conclude that \mbox{$\|\bigcdot\|_{H^{\frac{1}{2}}(\Gamma_D)}=\sup_{0 \neq f \in \widetilde{H}^{-\frac12}(\Gamma_D)}\frac{|f(\bigcdot)|}{\|f\|_{\widetilde{H}^{-\frac12}(\Gamma_D)}} \lesssim \sup_{0 \neq v \in H_{\triangle,0,\Gamma_N}(\Omega)} 
\frac{|\int_{\Gamma_D}\bigcdot \,\,\gamma_{\Gamma_N}^{\vec{n}}(\nabla v) \,ds|}{\|v\|_{H_\triangle(\Omega)}}$}.

Conversely, for arbitrary $v \in H_{\triangle,0,\Gamma_N}(\Omega)$ and $z \in H^1(\Omega)$, $|\int_{\partial \Omega} z \gamma_{\Gamma_D}^{\vec{n}}(\nabla v)\,ds|\leq \|v\|_{H_\triangle(\Omega)} \|z\|_{H^1(\Omega)}$, or 
$\sup_{0 \neq v \in H_{\triangle,0,\Gamma_N}(\Omega)} 
\frac{|\int_{\Gamma_D}\bigcdot \,\,\gamma_{\Gamma_D}^{\vec{n}}(\nabla v) \,ds|}{\|v\|_{H_\triangle(\Omega)}} \lesssim \|\bigcdot\|_{H^{\frac{1}{2}}(\Gamma_D)}$.
\end{proof}

Consequently, in \ref{secondaccent} and \ref{thirdaccent}, the equation $\gamma^{\vec{n}\,{\textstyle '}}_{\Gamma_D} \gamma_{\Gamma_D}(u)= \gamma^{\vec{n}\,{\textstyle '}}_{\Gamma_D}h_D$  in $H_{0,\Gamma_N}(\divv;\Omega)'$ can alternatively be replaced by
$(\gamma_{\Gamma_D}^{\vec{n}} \circ \nabla)'\gamma_{\Gamma_D}(u)=(\gamma_{\Gamma_D}^{\vec{n}} \circ \nabla)'h_D$  in $H_{\triangle,0,\Gamma_N}(\Omega)'$. This will be applied in Sect.~\ref{sec:applmodel}.

\subsection{Necessary inf-sup conditions}
To apply to \ref{secondaccent} or to \ref{thirdaccent} the  Least Squares discretisation \eqref{eq:minresprac}/\eqref{eq:4} or \eqref{eq:minresprac2}, 
one has to realize the following (uniform) inf-sup conditions.
\renewcommand{\themylistcounter}{(\alph{mylistcounter})} \begin{mylist}
\item \label{item:b} For \ref{secondaccent} and \ref{thirdaccent}:
Given a family of trial spaces $(U^\delta)_{\delta \in \Delta} \subset H^1(\Omega)$, we need a corresponding family of test spaces $(Y_{\rm{\ref{item:b}}}^\delta)_{\delta \in \Delta} \subset  H_{0,\Gamma_N}(\divv;\Omega)$ with
$$
\sup_{0 \neq \vec{v} \in Y_{\rm{\ref{item:b}}}^\delta} \frac{|\int_{\Gamma_D} \gamma_{\Gamma_D}(u) \gamma^{\vec{n}}_{\Gamma_D}(\vec{v})\,ds|}{\|\vec{v}\|_{H(\divv;\Omega)}} \gtrsim
\|\gamma^{\vec{n}\,{\textstyle '}}_{\Gamma_D}\gamma_{\Gamma_D}(u)\|_{H(\divv;\Omega)'} \quad(u \in U^\delta).
$$
\item \label{item:c} For \ref{secondaccent}: Given $(P^\delta)_{\delta \in \Delta} \subset H(\divv;\Omega)$, we need $(Y_{\rm{\ref{item:c}}}^\delta)_{\delta \in \Delta} \subset  H^1_{0,\Gamma_D}(\Omega)$
with 
$$
\sup_{0 \neq v \in Y_{\rm{\ref{item:c}}}^\delta} \frac{|\int_{\Gamma_N} \gamma_{\Gamma_N}^{\vec{n}}(\vec{p}) \gamma_{\Gamma_N}(v)\,ds|
}{\|v\|_{H^1(\Omega)}} 
\gtrsim
\| \gamma_{\Gamma_N}' \gamma_{\Gamma_N}^{\vec{n}}(\vec{p})\|_{(H^1_{0,\Gamma_D}(\Omega))'} \quad(\vec{p} \in P^\delta).
$$
\item \label{item:d} For \ref{thirdaccent}: Given $(P^\delta \times U^\delta)_{\delta \in \Delta} \subset L_2(\Omega)^d  \times H^1(\Omega)$, we need  $(Y_{\rm{\ref{item:d}}}^\delta)_{\delta \in \Delta} \subset  H^1_{0,\Gamma_D}(\Omega)$ with 
$$
\sup_{0 \neq v \in Y_{\rm{\ref{item:d}}}^\delta} \frac{|\int_\Omega \vec{p}\cdot\nabla v+Bu v\,dx|}{\|v\|_{H^1(\Omega)}}
\gtrsim
\sup_{0 \neq v \in H^1_{0,\Gamma_D}(\Omega)} \frac{|\int_\Omega \vec{p}\cdot\nabla v+Bu v\,dx|}{\|v\|_{H^1(\Omega)}} \quad(\vec{p},u) \in P^\delta \times U^\delta).
$$
\end{mylist}

In the following subsections, for the three cases \ref{item:b}--\ref{item:d} uniformly inf-sup stable pairs of finite element trial- and test spaces will be constructed.
For \ref{item:b} and \ref{item:c}, the test spaces $Y_{\rm{\ref{item:b}}}^\delta$ and $Y_{\rm{\ref{item:c}}}^\delta$ will be finite element spaces w.r.t.~conforming partitions of $\Omega$ whose intersections with $\Gamma_D$ or $\Gamma_N$ coincide with the underlying partition of the trial space, but which are maximally coarsened when moving away from $\Gamma_D$ or $\Gamma_N$, respectively. Although such highly non-uniform partitions give the smallest appropriate test spaces, clearly for convenience one may apply larger test spaces without jeopardizing the inf-sup stability.

Dependent on the formulation and the type of boundary condition (Dirichlet, Neumann, or mixed), an efficient iterative solution of the resulting Least Squares discretisations requires preconditioners for stiffness matrices of trial- or test- finite element spaces w.r.t.~scalar products on $H^1(\Omega)$, $H(\divv;\Omega)$, $H^1_{0,\Gamma_D}(\Omega)$, and $H_{0,\Gamma_N}(\divv;\Omega)$. Such preconditioners of linear complexity are available.

\subsection{Verification of inf-sup condition \ref{item:b}} \label{sec:b} Let $\Omega \subset \R^d$ be a polytope, $\tria^\delta$ be a conforming, (uniformly) shape regular partition of $\overline{\Omega}$ into (closed) $d$-simplices, and let $\cF(\tria^\delta)$ denote the set of (closed) facets of $K \in \tria^\delta$. Assume that $\Gamma_D$ is the union of some $e \in \cF(\tria^\delta)$. For some $q \in \N_0$, let 
\be \label{eq:19}
U^\delta:=\cS^{0}_{q+1}(\tria^\delta),
\ee
being the space of continuous piecewise polynomials of degree $q+1$ w.r.t.~$\tria^\delta$.

In order to prove inf-sup stability for the `original' mild and mild-weak formulations \ref{second} and \ref{third}, in \cite[Sect.~4.1]{204.19} a uniformly bounded `Fortin' interpolator
$\Pi_D^\delta \in \cL(\widetilde{H}^{-\frac12}(\Gamma_D),\widetilde{H}^{-\frac12}(\Gamma_D))$ has been constructed (there denoted by $\Pi_2^\delta$), with
 $\ran \Pi_D^\delta \subset \cS^{-1}_{q+1}(\cF(\tria^\delta) \cap \Gamma_D)$, being the space piecewise polynomials of degree $q+1$ w.r.t.~the partition $\cF(\tria^\delta) \cap \Gamma_D$, $\ran \gamma_{\Gamma_D}|_{U^\delta} \perp_{L_2(\Gamma_D)} \ran (\identity-\Pi_D^\delta)$, and 
$$
\|(\identity-{\Pi_D^\delta}') h_D\|_{H^{\frac12}(\Gamma_D)} \lesssim \sqrt{\sum_{e \in \cF(\tria^\delta) \cap \Gamma_D} \diam(e)^{2q+3} |h_D|_{H^{q+2}(e)}^2},
$$
for all $h_D \in H^{\frac12}(\Gamma_D)$ for which the right-hand side is finite.

To verify \ref{item:b}, we will construct a (uniformly) bounded right-inverse of the normal trace operator $\gamma_{\Gamma_D}^{\vec{n}}\in \cL(H_{0,\Gamma_N}(\divv;\Omega),\widetilde{H}^{-\frac12}(\Gamma_D)$ that maps $\cS^{-1}_{q+1}(\cF(\tria^\delta) \cap \Gamma_D)$ into a finite element space that will be used as the test space $Y_{\rm{\ref{item:b}}}^\delta$.

\begin{lemma} \label{lem:ext} Let $d \in \{2,3\}$. For \emph{any} (uniformly) shape regular partition $\tria_D^\delta$ of $\overline{\Omega}$ into (closed) $d$-simplices for which
$$
\cF(\tria_D^\delta) \cap \Gamma_D=   \cF(\tria^\delta) \cap \Gamma_D,
$$
there exists a linear
\be \label{eq:20}
E_D^\delta\colon \cS^{-1}_{q+1}(\cF(\tria^\delta) \cap \Gamma_D) \rightarrow Y_{\rm{\ref{item:b}}}^\delta:= \RT_{q+1}(\tria_D^\delta) \cap H_{0,\Gamma_N}(\divv;\Omega)
\ee
with $\gamma^{\vec{n}}_{\Gamma_D} E_D^\delta=\identity$, and 
$\sup_{\delta \in \Delta} \sup_{\{0\} \neq v \in \cS^{-1}_{q+1}(\cF(\tria^\delta) \cap \Gamma_D)}\frac{\|E_D^\delta v\|_{H(\divv;\Omega)}}{\|v\|_{\widetilde{H}^{-\frac12}(\Gamma_D)}}<\infty$.
 \end{lemma}
Here $\RT_{q+1}(\tria^\delta_D)\cap H_{0,\Gamma_N}(\divv;\Omega)$ denotes the space of Raviart-Thomas functions of order $q+1$ w.r.t.~$\tria^\delta_D$ whose normal components vanish at $\Gamma_N$.
Before we prove this lemma, we use it to demonstrate \ref{item:b}.
We set
$$
\Pi_{\rm{\ref{item:b}}}^\delta:=E_D^\delta \Pi_D^\delta \gamma^{\vec{n}}_{\Gamma_D}.
$$
Then $\Pi_{\rm{\ref{item:b}}}^\delta\colon H_{0,\Gamma_N}(\divv;\Omega) \rightarrow Y_{\rm{\ref{item:b}}}^\delta \subset H_{0,\Gamma_N}(\divv;\Omega) $ is uniformly bounded, and $\ran \gamma_{\Gamma_D}|_{U^\delta} \perp_{L_2(\Gamma_D)} \ran \gamma^{\vec{n}}_{\Gamma_D}(\identity-\Pi_{\rm{\ref{item:b}}}^\delta)$
as a consequence of $\gamma^{\vec{n}}_{\Gamma_D} \Pi_{\rm{\ref{item:b}}}^\delta=\Pi_D^\delta  \gamma^{\vec{n}}_{\Gamma_D}$, meaning that $\Pi_{\rm{\ref{item:b}}}^\delta$ is a valid Fortin interpolator. From  Theorem~\ref{thm:Fortin}  we conclude the following result.

\begin{theorem} \label{thm: infsup1} With $U^\delta$, $Y_{\rm{\ref{item:b}}}^\delta$ defined in \eqref{eq:19}-\eqref{eq:20}, the uniform inf-sup condition \ref{item:b} holds true.
\end{theorem}

\begin{remark}[data-oscillation] \label{rem:dataosc1} Because we are only interested in consistent data, the data-oscillation term to be estimated reads as $\|(\identity-{\Pi_{\rm{\ref{item:b}}}^\delta}')\gamma^{\vec{n}\,{\textstyle '}}_{\Gamma_D} h_D \|_{H_{0,\Gamma_N}(\divv;\Omega)'}$, where $h_D \in H^{\frac12}(\Gamma_D)$.
From $ \gamma^{\vec{n}}_{\Gamma_D} \Pi_{\rm{\ref{item:b}}}^\delta=\Pi_D^\delta  \gamma^{\vec{n}}_{\Gamma_D}$, we have 
\begin{align*}
\|(\identity-{\Pi_{\rm{\ref{item:b}}}^\delta}')\gamma^{\vec{n}\,{\textstyle '}}_{\Gamma_D} h_D\|_{H_{0,\Gamma_N}(\divv;\Omega)'}& =
\|\gamma^{\vec{n}\,{\textstyle '}}_{\Gamma_D} (\identity-{\Pi_D^\delta}') h_D\|_{H_{0,\Gamma_N}(\divv;\Omega)'}\\
&\leq \|(\identity-{\Pi_D^\delta}') h_D\|_{H^{\frac12}(\Gamma_D)}\\
&\leq \sqrt{\sum_{e \in \cF(\tria^\delta) \cap \Gamma_D} \diam(e)^{2q+3} |h_D|_{H^{q+2}(e)}^2},
\end{align*}
for all $h_D \in H^{\frac12}(\Gamma_D)$ for which the right-hand side is finite.
We conclude that the data-oscillation term is of order $q+\frac32$, which, as desired cf.~Remark~\ref{rem:2}, exceeds the order $q+1$ of best approximation of $U^\delta$ in $H^1(\Omega)$.
\end{remark}

\begin{proof}[Proof of Lemma~\ref{lem:ext}]
For $d \in \{2,3\}$, in \cite{70.991} a projector $Q_{\tria^\delta} \colon H_0(\divv;\Omega) \rightarrow \RT_{q+1}(\tria^\delta) \cap H_0(\divv;\Omega)$ has been constructed with the properties that
$\divv Q_{\tria^\delta}=\Pi_{\tria^\delta} \divv$, where $\Pi_{\tria^\delta}$ is the $L_2(\Omega)$-orthogonal projector onto $\cS^{-1}_{q+1}(\tria^\delta)$, and
\be \label{eq:7}
\begin{split}
&\|(\identity - Q_{\tria^\delta})\vec{v}\|_{L_2(K)}^2 \lesssim\\
& \sum_{\{K' \in \tria^\delta\colon K \cap K' \neq \emptyset\}} \hspace*{-1em}\min_{\hspace*{1em}\vec{q}_{K'} \in \RT_{q+1}(K')\hspace*{-1em}}\|\vec{v}-\vec{q}_{K'}\|_{L_2(K')}^2+\big(\tfrac{\diam(K')}{q+2}\big)^2\|(\identity - \Pi_{\tria^\delta})\divv \vec{v}\|_{L_2(K')}^2 \hspace*{-1em}
    \end{split}
\ee
($K \in \tria^\delta$, $\vec{v} \in H_0(\divv;\Omega)$), only dependent on $q$ and the shape regularity of $\tria^\delta$. 

Given $v \in \cS^{-1}_{q+1}(\cF(\tria^\delta) \cap \Gamma_D)$, let $\vec{w}^\delta \in Y_{\rm{\ref{item:b}}}^\delta$, $\vec{w} \in H_{0,\Gamma_N}(\divv;\Omega)$ be such that
$$
\gamma^{\vec{n}}_{\Gamma_D}(\vec{w}^\delta)=v=\gamma^{\vec{n}}_{\Gamma_D}(\vec{w}),
$$
where
$$
\|\vec{w}\|_{H(\divv;\Omega)} \lesssim \|v\|_{\widetilde{H}^{-\frac12}(\Gamma_D)}.
$$
Obviously $v \mapsto \vec{w}^\delta$ can be taken to be a linear map, and so can $v \mapsto \vec{w}$. For example, one may take $\vec{w}=\nabla z$ where $z$ solves $-\triangle z+z=0$ on $\Omega$, $\partial_{\vec{n}} z=v$ on $\Gamma_D$, $\partial_{\vec{n}} z=0$ on $\partial\Omega\setminus \Gamma_D$.

We define
$$
E_D^\delta v:=Q_{\tria_D^\delta} (\vec{w}-\vec{w}^\delta)+\vec{w}^\delta \in Y_{\rm{\ref{item:b}}}^\delta.
$$
It satisfies
$$
\gamma^{\vec{n}}_{\Gamma_D} (E_D^\delta v)=\gamma^{\vec{n}}_{\Gamma_D}(\vec{w}^\delta)=v,
$$
$\divv E_D^\delta v=\Pi_{\tria_D^\delta} \divv(\vec{w}-\vec{w}^\delta)+\divv \vec{w}^\delta=\Pi_{\tria_D^\delta} \divv \vec{w}$, and so
$$
\|\divv E_D^\delta v\|_{L_2(\Omega)} \leq \|\divv \vec{w}\|_{L_2(\Omega)} \lesssim  \|v\|_{\widetilde{H}^{-\frac12}(\Gamma_D)}.
$$
The proof is completed by
\begin{align*}
\|E_D^\delta v\|_{L_2(\Omega)^d} & \leq \|\vec{w}\|_{L_2(\Omega)^d}+ \|E_D^\delta v-\vec{w}\|_{L_2(\Omega)^d} \\
&= \|\vec{w}\|_{L_2(\Omega)^d}+\|(\identity -Q_{\tria_D^\delta})(\vec{w}^\delta-\vec{w})\|_{L_2(\Omega)^d} \\
&\lesssim 
\|\vec{w}\|_{L_2(\Omega)^d}+\|(\identity-\Pi_{\tria_D^\delta})\divv(\vec{w}^\delta-\vec{w})\|_{L_2(\Omega)}\\
&\hspace*{10em}+\sqrt{\sum_{K \in \tria^\delta} \min_{\vec{q}_{K} \in \RT_{q+1}(K)} \|\vec{w}^\delta-\vec{w}-\vec{q}_{K}\|_{L_2(K)}^2}\\
&=\|\vec{w}\|_{L_2(\Omega)^d}+\|(\identity-\Pi_{\tria_D^\delta})\divv \vec{w}\|_{L_2(\Omega)}\\
&\hspace*{10em}+\sqrt{\sum_{K \in \tria^\delta} \min_{\vec{q}_{K} \in \RT_{q+1}(K)} \|\vec{w}-\vec{q}_{K}\|_{L_2(K)}^2}\\
&\lesssim  \|\vec{w}\|_{H(\divv;\Omega)} \lesssim  \|v\|_{\widetilde{H}^{-\frac12}(\Gamma_D)}. \qedhere
\end{align*}
\end{proof}

Let $(\tria^\delta)_{\delta \in \Delta}$ be the family of all conforming partitions of $\Omega$ that can be constructed by Newest Vertex Bisection starting from an initial partition $\tria_0$ such that $\Gamma_D$ is the union of some $e \in \cF(\tria_0)$. Then a frugal way to construct $\tria_D^\delta$ is by a sequence of consecutive MARK and REFINE steps starting from $\tria^0$, where in each iteration only those simplices are marked for refinement that have an edge on $\Gamma_D$ that is not in $\cF(\tria^\delta) \cap \Gamma_D$.
The total number of simplices that will be marked is bounded by an absolute multiple of $\cF(\tria^\delta) \cap \Gamma_D$.
Consequently, an application of \cite[Thm.~2.4]{21} for $d=2$, or its extension \cite[Thm.~6.1]{249.87} for $d \geq 2$, shows that $\# \tria_D^\delta -\#\tria_0 \lesssim \# (\cF(\tria^\delta) \cap \Gamma_D)$. 
That is, the number of elements in the domain mesh $\tria_D^\delta$ (minus the number of elements in $\tria_0$) can be bounded by a constant multiple of the number of faces of elements in $\tria^\delta$ that are on $\Gamma_D$.

\subsection{Verification of inf-sup condition \ref{item:c}} \label{sec:c} 
Let $\Omega$ and $\tria^\delta$ be as in Sect.~\ref{sec:b}.
Let
\be \label{eq:21}
P^\delta :=\RT_q(\tria^\delta),
\ee
so that $\ran\gamma_{\Gamma_N}^{\vec{n}}|_{P^\delta} \subset  \cS^{-1}_q(\cF(\tria^\delta) \cap \Gamma_N)$.

In  \cite[Sect.~4.1]{204.19} a uniformly bounded projector 
$\Pi_N^\delta \in \cL(H_{00}^{\frac12} (\Gamma_N),H_{00}^{\frac12} (\Gamma_N))$ has been constructed (there denoted by $\Pi_1^\delta$), with
 $\ran \Pi_N^\delta \subset \cS^{0}_{d+q}(\cF(\tria^\delta) \cap \Gamma_N) \cap H^1_0(\Gamma_N)$, 
 $ \cS^{-1}_q(
\cF(\tria^\delta) \cap \Gamma_N) \perp_{L_2(\Gamma_N)} \ran (\identity-\Pi_N^\delta)$, and 
$$
\|(\identity-{\Pi_N^\delta}') h_N\|_{H^{-\frac12}(\Gamma_N)} \lesssim \sqrt{\sum_{e \in \cF(\tria^\delta) \cap \Gamma_N} \diam(e)^{2q+3} |h_N|_{H^{q+1}(e)}^2},
$$
for all $h_N \in H^{-\frac12}(\Gamma_N)$ for which the right-hand side is finite.

\begin{lemma} \label{lem:ext2} For \emph{any} (uniformly) shape regular partition $\tria_N^\delta$ of $\Omega$ into (closed) $d$-simplices such that
$$
\cF(\tria_N^\delta) \cap \Gamma_N=   \cF(\tria^\delta) \cap \Gamma_N,
$$
there exists a linear
\be \label{eq:22}
E_N^\delta\colon  \cS^{0}_{d+q}(\cF(\tria^\delta) \cap \Gamma_N) \cap H^1_0(\Gamma_N) \rightarrow Y_{\rm{\ref{item:c}}}^\delta:=\cS^0_{d+q}(\tria^\delta_N) \cap H^1_{0,\Gamma_D}(\Omega)
\ee
with $\gamma_{\Gamma_N} E_N^\delta=\identity$ and $\sup_{\delta \in \Delta} \sup_{0 \neq v \in  \cS^{0}_{d+q}(\cF(\tria^\delta) \cap \Gamma_N) \cap H^1_0(\gamma_N)}\frac{\|E_N^\delta v\|_{H^1(\Omega)}}{\|v\|_{H^{\frac12}_{00}(\Gamma_N)}}<\infty$.
\end{lemma}

\begin{proof} Given $v \in  \cS^{0}_{d+q}(\cF(\tria^\delta) \cap \Gamma_N) \cap H^1_0(\Gamma_N)$, let $z \in H^1_{0,\Gamma_D}(\Omega)$ solve $-\triangle z =0$ on $\Omega$, $z=v$ on $\Gamma_N$. Then $\|z\|_{H^1(\Omega)} \lesssim \|v\|_{H_{00}^{\frac12}(\Gamma_N)}$.
Now let $E_N^\delta v$ be the usual Scott-Zhang quasi-interpolant from \cite{247.2} of $z$ in $\cS^0_{d+q}(\tria^\delta_N)$, which interpolator is uniformly bounded in $H^1(\Omega)$ and preserves boundary data in $\cS^0_{d+q}(\cF(\tria^\delta_N) \cap \partial\Omega)$.
\end{proof}

The operator 
$$
\Pi_{\rm{\ref{item:c}}}^\delta:=E_N^\delta \Pi_N^\delta \gamma_{\Gamma_N} \colon H^1_{0,\Gamma_D}(\Omega) \rightarrow Y_{\rm{\ref{item:c}}}^\delta \subset H^1_{0,\Gamma_D}(\Omega)
$$
 is uniformly bounded, and $\ran \gamma_{\Gamma_N}^{\vec{n}}|_{P^\delta} \perp_{L_2(\Gamma_N)} \ran \gamma_{\Gamma_N}(\identity-\Pi_{\rm{\ref{item:c}}}^\delta)$
as a consequence of $ \gamma_{\Gamma_N} \Pi_{\rm{\ref{item:c}}}^\delta=\Pi_N^\delta  \gamma_{\Gamma_N}$.
From Theorem~\ref{thm:Fortin} we conclude the following result.
\begin{theorem} \label{thm: infsup2} With $P^\delta$, $Y_{\rm{\ref{item:c}}}^\delta$ defined in \eqref{eq:21}-\eqref{eq:22}, the uniform inf-sup condition \ref{item:c} holds true.
\end{theorem}

\begin{remark}[data-oscillation] \label{rem:dataosc2}
The data-oscillation term to be estimated reads as $\|(\identity-{\Pi_{\rm{\ref{item:c}}}^\delta}')\gamma_{\Gamma_N}' h_N \|_{H^1_{0,\Gamma_D}(\Omega)'}$ for $h_N \in H^{-\frac12}(\Gamma_N)$.
From $ \gamma_{\Gamma_N} \Pi_{\rm{\ref{item:c}}}^\delta=\Pi_N^\delta  \gamma_{\Gamma_N}$, we have 
\begin{align*}
\|(\identity-{\Pi_{\rm{\ref{item:c}}}^\delta}')\gamma_{\Gamma_N}' h_D\|_{{H^1_{0,\Gamma_D}(\Omega)}'}& =
\|\gamma_{\Gamma_N}' (\identity-{\Pi_N^\delta}') h_N\|_{{H^1_{0,\Gamma_D}(\Omega)}'}\\
&\leq \|(\identity-{\Pi_N^\delta}') h_N\|_{H^{-\frac12}(\Gamma_N)}\\
&\leq \sqrt{\sum_{e \in \cF(\tria^\delta) \cap \Gamma_N} \diam(e)^{2q+3} |h_N|_{H^{q+1}(e)}^2},
\end{align*}
for all $h_N \in H^{-\frac12}(\Gamma_D)$ for which the right-hand side is finite.
We conclude that the data-oscillation term is of order $q+\frac32$, which exceeds the order $q+1$ of best approximation of $P^\delta=\RT_{q}(\tria^\delta)$ in $H(\divv;\Omega)$.
\end{remark}

Finally, as we have seen in Sect.~\ref{sec:b}, $\tria_N^\delta$ can be constructed with $\# \tria_N^\delta -\#\tria_0 \lesssim \# (\cF(\tria^\delta) \cap \Gamma_N)$.

\subsection{Verification of inf-sup condition \ref{item:d}}  \label{sec:d} Let $\Omega$, $\tria^\delta$, and $U^\delta=\cS^0_{q+1}(\tria^\delta)$ be as in Sect.~\ref{sec:b}. Take $P^\delta:=\cS^{-1}_q(\tria^\delta)^d$, and assume that $\ran B|_{U^\delta} \subset \cS_{q}^{-1}(\tria^\delta)$. Writing
$$
\int_\Omega \vec{p}\cdot\nabla v+Bu v\,dx=\sum_{K \in \tria^\delta} \int_K (Bu-\divv \vec{p})v\,dx+\int_{\partial K} v \vec{p}\cdot\vec{n}\,ds,
$$
the arguments used for the construction in $\Pi_1^\delta$ in  \cite[Sect.~4.1]{204.19}  show that \ref{item:d} is satisfied for $Y_{\rm{\ref{item:d}}}^\delta := \cS_{q+d+1}^0(\tria^\delta) \cap H^1_{0,\Gamma_D}(\Omega)$. 

As follows from \cite[Rem.~4.2]{204.19}, for $g \in H^1_{0,\Gamma_d}(\Omega)'$ such that for all $K \in \tria^\delta$, $g|_K \in H^{q+1}(K)$, with this $Y_{\rm{\ref{item:d}}}^\delta$ the data-oscillation term is order $q+2$, which exceeds the order $q+1$ of best approximation of $P^\delta \times U^\delta$
 in $L_2(\Omega)^d \times H^1(\Omega)$.

\section{Another application: Stokes equations} \label{sec:5}
Our approach to append inhomogeneous essential boundary conditions to a Least Squares functional is not restricted to elliptic problems of second order.
As an example we consider the Stokes equations.
On a bounded Lipschitz domain $\Omega \subset \R^d$, where $d =3$,\footnote{By reading $\curl$ as the scalar-valued operator $\vec{\omega} \mapsto \partial_1 \omega_2-\partial_2 \omega_1$, the results also hold for $d=2$.} we seek $(\vec{u},p)$ that, for given data $\vec{f}$, $g$, and $\vec{h}$, and  viscosity $\nu >0$, satisfy
\be \label{eq:Stokesequations}
 \left\{
\begin{array}{r@{}c@{}ll}
-\nu \triangle \vec{u}+ \nabla p &\,\,=\,\,& \vec{f} &\text{ on } \Omega,\\
\divv \vec{u} &\,\,=\,\,& g &\text{ on } \Omega,\\
\vec{u} &\,\,=\,\,& \vec{h} &\text{ on } \partial\Omega.
\end{array}
\right.
\ee

We consider its first order velocity-pressure-vorticity formulation.

\begin{proposition} \label{prop:stokes} With $\gamma\colon H^1(\Omega)^d \rightarrow H^{\frac12}(\Omega)^d\colon \vec{w}\rightarrow \vec{w}|_{\partial\Omega}$, let
$$
G(\vec{u},p,\vec{\omega}):=\big(\vec{v}\mapsto \int_\Omega \vec{\omega} \cdot \curl \vec{v}-\tfrac{p}{\nu}\divv \vec{v}\,dx, \divv \vec{u},\vec{\omega}-\curl \vec{u},  \gamma(\vec{u})\big).
$$
Then
$$
G \!\in\! \Lis\big(H^1(\Omega)^d\!\times\! L_{2,0}(\Omega) \!\times\! L_2(\Omega)^{2d-3}, H^{-1}(\Omega)^d \!\times\! L_{2,0}(\Omega) \!\times\! L_2(\Omega)^{2d-3} \!\times\! H^{\frac12}(\partial\Omega)^d \big),
$$
and
$$
G(\vec{u},p,\vec{\omega}) = \big(\vec{v}\mapsto\int_\Omega \tfrac{\vec{f} \cdot \vec{v}}{\nu} -g\divv \vec{v}\,dx, g,0, \vec{h}\big) \\
$$
is a consistent formulation of \eqref{eq:Stokesequations}.
In particular,
$$
\|\vec{u}\|_{H^1(\Omega)}^2\!+\!\tfrac{1}{\nu^2}\|p\|_{L_{2,0}(\Omega)}^2\!+\!\|\vec{\omega}\|_{L_2(\Omega)}^2\!\eqsim \!
\|G(\vec{u},p,\vec{\omega})\|^2_{H^{-1}(\Omega)^d \times L_{2,0}(\Omega) \times L_2(\Omega)^{2d-3} \times H^{\frac12}(\partial\Omega)^d}
$$
uniformly in $\nu>0$. 
\end{proposition}

This proposition generalizes \cite{35.93006} which considers $g=0$ and $\vec{h}=0$ (see also \cite[Ch.~7]{23.5}).
Its proof is postponed to Appendix~\ref{sec:app}.

Analogously to Sect.~\ref{sec:field}, using \eqref{eq:15} we have the following corollary.

\begin{corollary} \label{prop:stokes2} With $\gamma^{\vec{n}}\colon H(\divv;\Omega)^d\rightarrow H^{-\frac12}(\partial\Omega)^d\colon (\vec{v}_i)\mapsto(\vec{v}_i|_{\partial\Omega}\cdot \vec{n})$, let 
$$
G(\vec{u},p,\vec{\omega}):=\big(\vec{v}\mapsto \int_\Omega \vec{\omega} \cdot \curl \vec{v}-\tfrac{p}{\nu}\divv \vec{v}\,dx, \divv \vec{u},\vec{\omega}-\curl \vec{u}, {\gamma^{\vec{n}}}'\gamma(\vec{u})\big).
$$
Then
$$
G \!\in\! \cL\big(H^1(\Omega)^d\!\times\! L_{2,0}(\Omega) \!\times\! L_2(\Omega)^{2d-3}, H^{-1}(\Omega)^d \!\times\! L_{2,0}(\Omega) \!\times\! L_2(\Omega)^{2d-3} \!\times\! (H(\divv;\Omega)^d)' \big)
$$
is a homeomorphism with its range, and 
\be \label{eq:12}
G(\vec{u},p,\vec{\omega}) = \big(\vec{v}\mapsto\int_\Omega \tfrac{\vec{f} \cdot \vec{v}}{\nu} -g\divv \vec{v}\,dx, g,0,{\gamma^{\vec{n}}}' \vec{h}\big) \\
\ee
is a consistent formulation of \eqref{eq:Stokesequations}.
In particular,
$$
\|\vec{u}\|_{H^1(\Omega)}^2\!+\!\tfrac{1}{\nu^2}\|p\|_{L_{2,0}(\Omega)}^2\!+\!\|\vec{\omega}\|_{L_2(\Omega)}^2\!\eqsim \!
\|G(\vec{u},p,\vec{\omega})\|^2_{H^{-1}(\Omega)^d \times L_{2,0}(\Omega) \times L_2(\Omega)^{2d-3} \times (H(\divv;\Omega)^d)'}
$$
uniformly in $\nu>0$. 
\end{corollary}

To apply to \eqref{eq:12} the Least Squares discretisation \eqref{eq:minresprac}/\eqref{eq:4} or \eqref{eq:minresprac2}, one has to realize the following inf-sup conditions.
\renewcommand{\themylistcounter}{(\Roman{mylistcounter})} \begin{mylist}
\item \label{item:I} Given families $(P^\delta)_{\delta \in \Delta} \subset L_{2,0}(\Omega)$ and $(C^\delta)_{\delta \in \Delta} \subset L_2(\Omega)^{2d-3}$, one needs a family $(Y^\delta)_{\delta \in \Delta} \subset H^1_0(\Omega)^d$ with 
$$
\sup_{0 \neq v \in Y^\delta} \frac{|M(p,\vec{\omega})(\vec{v})|}{\|\vec{v}\|_{H^1(\Omega)^d}} \gtrsim \|M(p,\vec{\omega})(\vec{v})\|_{H^{-1}(\Omega)^d} \quad ((p,\vec{\omega}) \in P^\delta \times C^\delta),
$$
where $M(p,\vec{\omega})(\vec{v}):=\int_\Omega \vec{\omega} \cdot \curl \vec{v}-\tfrac{p}{\nu}\divv \vec{v}\,dx$.
\item \label{item:II}
Given a family $(U^\delta)_{\delta \in \Delta} \subset H^1(\Omega)^d$, one needs $(Y^\delta)_{\delta \in \Delta} \subset H(\divv;\Omega)^d$ with 
$$
\sup_{0 \neq \vec{v} \in Y^\delta} \frac{|\int_{\partial\Omega} \gamma(\vec{u}) \cdot \gamma^{\vec{n}}(\vec{v}) \,ds|}{\|\vec{v}\|_{H(\divv;\Omega)^d}} \gtrsim
\|{\gamma^{\vec{n}}}'\gamma(\vec{u})\|_{(H(\divv \Omega)^d)'}\quad(\vec{u} \in U^\delta).
$$
\end{mylist}

Concerning \ref{item:I}: For $\tria^\delta$ being as in Sect.~\ref{sec:b}, let $P^\delta \subset \cS^{-1}_q(\tria^\delta) \cap L_{2,0}(\Omega)$ and $C^\delta \subset \cS^{-1}_q(\tria^\delta)^d$.
Then by writing for $(p,\vec{\omega}) \in P^\delta \times C^\delta$,
$$
M(p,\vec{\omega})(\vec{v})=\sum_{K \in \tria^\delta} \int_K (\curl \omega +\tfrac{1}{\nu} \nabla p)\cdot \vec{v}\,dx-\int_{\partial K} \tfrac{p}{\nu} \vec{v}\cdot\vec{n}+(\vec{\omega} \times \vec{v})\cdot \vec{n}\,ds,
$$
the arguments used for the construction in $\Pi_1^\delta$ in  \cite[Sect.~4.1]{204.19}  show that \ref{item:I} is satisfied for $Y^\delta := \cS_{q+d}^0(\tria^\delta)^d \cap H^1_{0}(\Omega)^d$.

As follows from \cite[Rem.~4.2]{204.19}, by taking $Y^\delta = \cS_{q+d+1}^0(\tria^\delta)^d \cap H^1_{0}(\Omega)^d$,
for $f \in H^{q+1}(\Omega)$ and $g \in H^{q+2}(\Omega)$ data-oscillation is of order $q+2$ which exceeds the order $q+1$ of best approximation of $P^\delta$ in $L_{2,0}(\Omega)$ and $C^\delta$ in $L_{2}(\Omega)^d$.

Concerning \ref{item:II}:  This inf-sup condition has been discussed in Sect.~\ref{sec:b} for the `scalar case'.
The results given there show that if $U^\delta=\cS^0_{q+1}(\tria^\delta)^d$, and $\tria_{\partial\Omega}^\delta$ is \emph{some} uniformly shape regular partition of $\Omega$ into (closed) $d$-simplices with $\cF(\tria_{\partial\Omega}^\delta) \cap \partial\Omega =\cF(\tria^\delta) \cap \partial\Omega$, then  \ref{item:II} is satisfied for $Y^\delta=\RT_{q+1}(\tria_{\partial\Omega}^\delta)^d$.

Similar to Remark~\ref{rem:dataosc1}, for $\vec{h}|_{e} \in H^{q+2}(e)^d$ ($e \in \cF(\tria^d) \cap \partial\Omega$), data-oscillation is of order $q+\frac32$ which exceeds the order $q+1$ of best approximation of $U^\delta$ in $H^1(\Omega)$.

\section{Application with Machine Learning} \label{sec:6}
\subsection{Abstract setting} \label{sec:abstract} \label{sec:6.1}
We return to the abstract setting to approximate the solution of $Gu=f$, where $G \in \cL(X,Y')$ with $\|\bigcdot\|_X \eqsim \|G \bigcdot\|_{Y'}=:\nrm \bigcdot\nrm_X$, and $f \in \ran G$.
Given a \emph{set} $\cX \subset X$, we aim to minimize $\frac12\|f-Gw\|_{Y'}^2$ over $w \in \cX$. As set $\cX$ we have in mind a collection of (Deep) Neural Net functions.

Recall the problem that in most applications  $\|\bigcdot\|_{Y'}$ is not evaluable. As before, to analyze this situation it suffices to consider the setting that $Y'=Y_1'\times Y_2'$ with $\|\bigcdot\|_{Y_2'}$ evaluable, and $\|\bigcdot\|_{Y_1'}$ not being evaluable.

Earlier, for $X^\delta$ being a closed \emph{linear subspace} of $X$, we solved this problem by replacing $\|\bigcdot\|_{Y_1'}$ by the discretized dual norm
$\sup_{0 \neq  v \in Y_1^\delta}\frac{|\bigcdot(v)|}{\|v\|_{Y_1}}$, where $Y_1^\delta$ is a closed \emph{linear subspace} of $Y_1$ that satisfies the (uniform) inf-sup condition \eqref{eq:defgamma}. As has been shown in Theorem~\ref{thm:quasi-opt}, the resulting Least-Squares approximation $u^\delta \in X^\delta$ is a quasi-best approximation to $u$ w.r.t.~$\nrm\bigcdot\nrm_X$.

For \emph{subsets} $\cX \subset X$ and $\cY_1 \subset Y_1$, a substitute for Theorem~\ref{thm:quasi-opt} is the following Proposition~\ref{prop:2}.
It requires that $\cY_1$ is closed under scalar multiplication, which, by the absence of an activation function in the output layer, holds true for $\cY_1$ being a collection of `adversarial' (Deep) Neural Net functions (possibly with components multiplied by a function $\phi$ with $\phi(x) $ proportional to the distance of $x$ to (part of) the boundary to enforce an homogeneous boundary condition) .

Proposition~\ref{prop:2} is based on \cite[Lemma~3]{20.25}, where \eqref{prop:2} is milder than the corresponding condition in \cite{20.25}.

\begin{proposition} \label{prop:2}
Given a set $\cX \subset X$, let $\cY_1 \subset Y_1$ be closed under scalar multiplication, and sufficiently large such that
\be \label{eq:18}
\alpha:=\inf_{\{w,\tilde{w} \in \cX\colon G_1(w-\tilde{w})\neq 0\}} \frac{\sup_{0 \neq v \in \cY_1} \frac{|G_1(w-\tilde{w})(v)|}{\|v\|_{Y_1}}}{\|G_1(w-\tilde{w})\|_{Y_1'}}>0.
\ee
Then the Least Squares approximation\footnote{For simplicity we assume that a minimum exists. Otherwise, for $(u_n)_n \subset \cX$ with 
$\lim_n \sup_{0 \neq v \in \cY_1} \frac{|(f_1-G_1 u_n)(v)|^2}{\|v\|_{Y_1}^2} +\|f_2-G_2 u_n\|_{Y_2'}^2 =
\inf_{w \in \cX} \sup_{0 \neq v \in \cY_1} \frac{|(f_1-G_1 w)(v)|^2}{\|v\|_{Y_1}^2} +\|f_2-G_2 w\|_{Y_2'}^2$,
one verifies that $\limsup_n  \nrm u-u_n\nrm_X \leq  (1+\tfrac{2}{\alpha}) \inf_{w \in \cX} \nrm u-w\nrm_X$.} 
$$
u_\cX(=u_{\cX \cY_1})=\argmin_{w \in \cX} \tfrac12 \big\{\sup_{0 \neq v \in \cY_1} \frac{|(f_1-G_1 w)(v)|^2}{\|v\|_{Y_1}^2} +\|f_2-G_2 w\|_{Y_2'}^2 \big\}
$$
satisfies
$$
 \nrm u-u_\cX\nrm_X \leq  (1+\tfrac{2}{\alpha}) \inf_{w \in \cX} \nrm u-w\nrm_X.
$$
\end{proposition}

\begin{proof}  Given an $\eps>0$, for any $w,\tilde{w} \in \cX$, there exists a $v_1\in \cY_1$ with $\|v_1\|_{Y_1}=\|G_1 (w-\tilde{w})\|_{Y_1'}$ and $(G_1 (w-\tilde{w}))(v_1) \geq (\alpha-\eps) \|G_1 (w-\tilde{w})\|_{Y_1'}^2$, where we used that $\cY_1$ is closed under scalar multiplication. Also, there exists a $v_2\in Y_2$ with $\|v_2\|_{Y_2}=\|G_2 (w-\tilde{w})\|_{Y_2'}$ and $(G_2 (w-\tilde{w}))(v_2) = \|G_2 (w-\tilde{w})\|_{Y_2'}^2$.
So with $v=(v_1,v_2) \in \cY_1 \times Y_2$,
\begin{align*}
(G(w-\tilde{w}))(v)=\sum_{i=1}^2 (G_i(w-\tilde{w}))(v_i) & \geq (\alpha-\eps) \sum_{i=1}^2 \|G_i (w-\tilde{w})\|_{Y_i'}^2\\ &=(\alpha-\eps) \nrm w-\tilde{w}\nrm_X \|v\|_{\cY},
\end{align*}
and so
$$
\inf_{w \neq \tilde{w} \in \cX} \sup_{0 \neq v \in \cY_1 \times Y_2} \frac{|(G(w-\tilde{w}))(v)|}{\|v\|_{Y}\nrm w-\tilde{w}\nrm_X} \geq \alpha.
$$

Now for any $w \in \cX$,
\begin{align*}
\nrm u-u_\cX\nrm_X &\leq \nrm u-w\nrm_X+\nrm w-u_\cX\nrm_X\\
&\leq \nrm u-w\nrm_X+\tfrac{1}{\alpha} \sup_{0 \neq v \in \cY_1 \times Y_2} \frac{|(G(w-u_\cX))(v)|}{\|v\|_{Y}} \\
&\leq \nrm u-w\nrm_X+\tfrac{1}{\alpha} \big(\sup_{0 \neq v \in \cY_1 \times Y_2} \frac{|(f-Gw)(v)|}{\|v\|_{Y}} +\sup_{0 \neq v \in \cY_1 \times Y_2} \frac{|(f-Gu_\cX)(v)|}{\|v\|_{Y}} \big) \\
&\leq \nrm u-w\nrm_X+\tfrac{2}{\alpha}\sup_{0 \neq  v \in \cY_1 \times Y_2} \frac{|(f-Gw)(v)|}{\|v\|_{Y}} \\
&\leq (1+\tfrac{2}{\alpha}) \nrm u-w\nrm_X,
\end{align*}
where the one but last inequality holds true by definition of $u_\cX$ and Footnote~\ref{foottie}.
\end{proof}

\begin{remark}[A posteriori error estimation]
In the setting of Proposition~\ref{prop:2}, we set the (squared) error estimator by
$$
\cE(w,f):=\sup_{0 \neq v \in \cY_1} \frac{|(f_1-G_1 w)(v)|^2}{\|v\|_{Y_1}^2} +\|f_2-G_2 w\|_{Y_2'}^2, 
$$
Clearly it holds that $\cE(w,f) \leq \nrm u- w\nrm_X^2$, i.e., the estimator is \emph{efficient}.

Now let $\cX \subset \widetilde{\cX} \subset X$ be such that  $(\widetilde{\cX},\cY_1)$ is inf-sup stable in the sense of \eqref{eq:18} with constant $\widetilde{\alpha} \in (0,\alpha]$, and such that for some constant $\varrho<1$,
$$
\inf_{z \in \widetilde{\cX}} \nrm u -z\nrm_X \leq \varrho \inf_{z \in \cX} \nrm u -z\nrm_X,
$$
known as  a \emph{saturation assumption}. Then, as in the proof of Proposition~\ref{prop:2}, we have
$$
\nrm u_\cX-u_{\widetilde{\cX}}\nrm_X \leq \tfrac{2}{\widetilde{\alpha} }\sup_{0 \neq  v \in \cY_1 \times Y_2} \frac{|(f-Gu_\cX)(v)|}{\|v\|_{Y}}=\tfrac{2}{\widetilde{\alpha} } \sqrt{\cE(u_\cX,f)}.
$$
From $\nrm u-u_\cX\nrm_X \leq \nrm u-u_{\widetilde{\cX}} \nrm_X+\nrm u_{\widetilde{\cX}}-u_\cX\nrm_X \leq \varrho \nrm u-u_\cX \nrm_X+\nrm u_{\widetilde{\cX}}-u_\cX\nrm_X$, we conclude that $\nrm u-u_\cX\nrm_X  \leq \tfrac{2}{(1-\varrho)\widetilde{\alpha} }\sqrt{\cE(u_\cX,f)}$, i.e., the estimator is \emph{reliable}.

For the case that $\cX$, $\widetilde{\cX}$, and $\cY_1$ are {\em linear subspaces}, a similar technique to construct an efficient and reliable a posteriori estimator was applied in \cite[Lemma 2.6]{168.88}.
\end{remark}

A straighforward (approximate) computation of $\sup_{0 \neq v \in \cY_1} \frac{|(f_1-G_1 w)(v)|^2}{\|v\|_{Y_1}^2}$, required for the Least Squares approximation (as well as for the a posteriori error estimator), turns out to be unstable as can be understood from the fact that for $f_1-G_1 w=0$, any $0 \neq v \in \cY_1$ is a supremizer.
A stable computation is provided by the following result.
\begin{lemma}[{\cite[Lemma~4]{20.25}}] \label{lem:1}
Let $\cY_1 \subset Y_1$ be closed under scalar multiplication. Then for any $g \in Y_1'$,
$$
\sup_{0 \neq v \in \cY_1} \frac{|g(v)|^2}{\|v\|_{Y_1}^2}=\sup_{v \in \cY_1} 2\Re g(v) -\|v\|_{Y_1}^2. \,\,\,\footnotemark
$$
\footnotetext{Obviously, taking the real part can be omitted if $Y_1$ is a Hilbert space over $\R$ as we will consider in our applications.}
\end{lemma}

\begin{proof} Let us denote $\sup_{0 \neq v \in \cY_1} \frac{|g(v)|}{\|v\|_{Y_1}}$ by $\|g\|_{\cY_1'}$.
Given $\eps>0$, let $\widetilde{v}\in \cY_1$ be such that $\|\widetilde{v}\|_{Y_1}=1$ and $g(\widetilde{v}) \geq (1-\tfrac{\eps}{2}) \|g\|_{\cY_1'}$. Then for $v:=\|g\|_{\cY_1'} \widetilde{v}$, 
$$
2g(v)-\|v\|_{Y_1}^2=2\|g\|_{\cY_1'}g(\widetilde{v})-\|g\|_{\cY_1'}^2 \geq (1-\eps) \|g\|_{\cY_1'}^2,
$$
so that $\|g\|_{\cY_1'}^2 \leq \sup_{v \in \cY_1} 2\Re g(v) -\|v\|_{Y_1}^2$.

On the other hand,
$$
\sup_{v \in \cY_1} 2 \Re g(v) -\|v\|_{Y_1}^2
\leq \sup_{v \in \cY_1} 2\|g\|_{\cY_1'}\|v\|_{Y_1}-\|v\|_{Y_1}^2 \leq \|g\|_{\cY_1'}^2. \qedhere
$$
\end{proof}

Above results show how to avoid the unfeasible computation of $\|\bigcdot\|_{Y_1'}$. 
When also the computation of $\|\bigcdot\|_{Y_1}$ is unfeasible, for $\cY_1 \subset Y_1$ being a \emph{linear subspace} the approach from \cite{204.19}, which was recalled in Sect.~\ref{sec:implementation}, is to replace $\|\bigcdot\|_{Y_1}$ by an on $\cY_1$ equivalent norm defined in terms of an (optimal) preconditioner.
This approach does \emph{not} apply in the current setting, and so we will avoid the situation that both $\|\bigcdot\|_{Y_1'}$ and $\|\bigcdot\|_{Y_1}$ are non-evaluable norms, as when $Y_1$ is a fractional Sobolev norm. 

\subsection{Application to model elliptic second order boundary value problem} \label{sec:applmodel}
For the model elliptic second order boundary value problem \eqref{bvp}, we apply Least Squares to the \emph{modified mild}, or \emph{modified mild-weak} first order system formulations \ref{secondaccent} or \ref{thirdaccent} from Sect.~\ref{sec:field}, where we replace the imposition of the Dirichlet boundary condition by means of \eqref{eq:15} by that from Lemma~\ref{lem:extra}.
For the formulation \ref{secondaccent}, for $\cX \subset X:=H(\divv;\Omega) \times H^1(\Omega)$ and $\cY \subset H_{\triangle,0,\Gamma_N}(\Omega) \times H^1_{0,\Gamma_D}(\Omega)$, using Lemma~\ref{lem:1} it results in the problem of finding
\be \label{eq:103}
\begin{split}
\argmin_{(\vec{q},w) \in \cX } &\Big[\tfrac12
\|\vec{q} -A\nabla w\|_{L_2(\Omega)^d}^2+\tfrac12 \|B w-\divv \vec{q}-g\|_{L_2(\Omega)}^2+
\\
&\sup_{\vec{v}=(v_1,v_2) \in \cY} \int_{\Gamma_D}(w-h_D) \gamma_{\Gamma_D}^{\vec{n}}(\nabla v_1)\,ds +\int_{\Gamma_N} (\vec{q}\cdot\vec{n}-h_N)v_2\,ds
\\
&\hspace*{12em}-\tfrac12\|v_1\|_{H_\triangle(\Omega)}^2-\tfrac12\|v_2\|_{H^1(\Omega)}^2
\Big].
\end{split}
\ee
Obvious adaptations are required when either $\Gamma_D=\emptyset$ or $\Gamma_N=\emptyset$ (so that $\cY \subset H^1(\Omega)$ or $\cY \subset H_{\triangle}(\Omega)$).

Analogously, one derives the Least Squares problem resulting from \ref{thirdaccent} with \eqref{eq:15} replaced by Lemma~\ref{lem:extra}.

For large $d$, the approximation of the $d$-dimensional vector field $\vec{q}$ is computational demanding. In that case one may resort to the standard \emph{second order} variational formulation. From Lemma~\ref{lem:extra} one infers that finding $u \in H^1(\Omega)$ such that
$$
 \left\{
\begin{array}{r@{}c@{}ll}
\int_\Omega A \nabla u\cdot \nabla v+B u\,v\,dx &\,\,=\,\,& g(v) +\int_{\Gamma_N} h_N \gamma_{\Gamma_N}(v) \,ds& (v \in H^1_{0,\Gamma_D}(\Omega)),\\
(\gamma_{\Gamma_D}^{\vec{n}}\circ \nabla)' \gamma_{\Gamma_D}(u)&\,\,=\,\,& (\gamma_{\Gamma_D}^{\vec{n}}\circ \nabla)' h_D & \text{ in } H_{\triangle,0,\Gamma_N}(\Omega)'.
\end{array}
\right.
$$
defines a homeomorphism between $H^1(\Omega)$ and its range in $(H^1_{0,\Gamma_D}(\Omega) \times H_{\triangle,0,\Gamma_N}(\Omega))'$.
For $\cX \subset X:=H^1(\Omega)$ and $\cY \subset H^1_{0,\Gamma_D}(\Omega) \times H_{\triangle,0,\Gamma_N}(\Omega)$, it leads to the Least Squares problem of finding 
\be \label{eq:104}
\begin{split}
\argmin_{w \in \cX} &\sup_{\vec{v}=(v_1,v_2) \in \cY}
\int_\Omega A \nabla w \cdot\nabla v_1+B w\, v_1\,dx-g(v_1)-\int_{\Gamma_N} h_N \gamma_{\Gamma_N}(v_1) \,ds\\
&+\int_{\Gamma_D} (w-h_D) \gamma_{\Gamma_D}^{\vec{n}} (\nabla v_2)\,ds-\tfrac12 \|v_1\|_{H^1(\Omega)}^2 -\tfrac12\|v_2\|_{H_{\triangle}(\Omega)}^2.
\end{split}
\ee

For this second order formulation, and, in the case of mixed boundary conditions, for the modified mild first order formulation, one has to enforce homogeneous boundary conditions in the test set by multiplying Neural Net functions by a function that is proportional to the distance to the corresponding part of the boundary.

Proposition~\ref{prop:2} shows that \emph{if} in the above formulations $\cY$ is sufficiently large in relation to $\cX$, i.e., independently from the data $(g,h_D,h_N)$, \emph{then} the Least Squares solution from $\cX$ is a \emph{quasi-best approximation} from $\cX$ to the exact solution in the norm on $X$.

A similar conclusion can be drawn for the Least Squares approximation to the Stokes equations based on their formulation from Corollary~\ref{prop:stokes2}.

\begin{remark}
In the above examples we have seen that for sufficiently large $\cY$, the Least Squares solution from $\cX$ is a quasi-best approximation to the exact solution from $X$. Notice, however, that other than for the finite element setting discussed in Sections~\ref{sec:b}-\ref{sec:d}, and ~\ref{sec:5}, so far for sets of Neural Net functions $\cX$ and $\cY$ the condition of $\cY$ being sufficiently large, i.e., to satisfy \eqref{eq:18}, has \emph{not} been verified.
\end{remark}

\section{Numerical results} \label{sec:7}
\subsection{Experiments with Finite Elements} \label{sec:numfem}
We take an example from \cite{249.965}. On a 
\parbox[top]{9cm}{rectangular domain $\Omega = (-1,1)\times (0,1)$, i.e., $d=2$, with  Neumann and Dirichlet boundaries $\Gamma_N=[-1,0]\times\{0\}$ and $\Gamma_D=\overline{\partial \Omega \setminus \Gamma_N}$,
for $g \in H^1_{0,\Gamma_D}(\Omega)'$, $h_D \in H^{\frac12}(\Gamma_D)$, and $h_N \in H^{-\frac12}(\Gamma_N)$, we consider the Poisson problem of finding $u \in H^1(\Omega)$ that satisfies} \qquad\raisebox{-0.75cm}[0cm][0cm]{\begin{picture}(0,0)%
\includegraphics{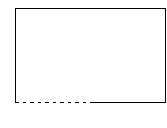}%
\end{picture}%
\setlength{\unitlength}{1973sp}%
\begingroup\makeatletter\ifx\SetFigFont\undefined%
\gdef\SetFigFont#1#2#3#4#5{%
  \reset@font\fontsize{#1}{#2pt}%
  \fontfamily{#3}\fontseries{#4}\fontshape{#5}%
  \selectfont}%
\fi\endgroup%
\begin{picture}(2652,2028)(961,-1957)
\put(1576,-1861){\makebox(0,0)[lb]{\smash{{\SetFigFont{8}{9.6}{\rmdefault}{\mddefault}{\updefault}{\color[rgb]{0,0,0}$\Gamma_N$}%
}}}}
\put(2176,-886){\makebox(0,0)[lb]{\smash{{\SetFigFont{8}{9.6}{\rmdefault}{\mddefault}{\updefault}{\color[rgb]{0,0,0}$\Omega$}%
}}}}
\put(3526,-1861){\makebox(0,0)[lb]{\smash{{\SetFigFont{8}{9.6}{\rmdefault}{\mddefault}{\updefault}{\color[rgb]{0,0,0}$1$}%
}}}}
\put(976,-1561){\makebox(0,0)[lb]{\smash{{\SetFigFont{8}{9.6}{\rmdefault}{\mddefault}{\updefault}{\color[rgb]{0,0,0}$0$}%
}}}}
\put(976,-136){\makebox(0,0)[lb]{\smash{{\SetFigFont{8}{9.6}{\rmdefault}{\mddefault}{\updefault}{\color[rgb]{0,0,0}$1$}%
}}}}
\put(1051,-1861){\makebox(0,0)[lb]{\smash{{\SetFigFont{8}{9.6}{\rmdefault}{\mddefault}{\updefault}{\color[rgb]{0,0,0}$-1$}%
}}}}
\put(2326,-1861){\makebox(0,0)[lb]{\smash{{\SetFigFont{8}{9.6}{\rmdefault}{\mddefault}{\updefault}{\color[rgb]{0,0,0}$0$}%
}}}}
\end{picture}%
}
$$
 \left\{
\begin{array}{r@{}c@{}ll}
-\Delta u&\,\,=\,\,& g &\text{ on } \Omega,\\
u &\,\,=\,\,& h_D &\text{ on } \Gamma_D,\\
\nabla u \cdot \vec{n}&\,\,=\,\,& h_N &\text{ on } \Gamma_N.
\end{array}
\right.
$$
We prescribe the solution $u(r,\theta) := r^{\frac12} \sin \frac{\theta}{2}$ in polar coordinates, and determine the data correspondingly. Then $g=0$, $h_N=0$, and $h_D=0$ on $[0,1]\times\{0\}$, but $h_D \neq 0$ on the remaining part of $\Gamma_D$.
It is known that $u \in H^{\frac32-\eps}(\Omega)$ for all $\eps>0$, but $u \not\in H^{\frac32}(\Omega)$ (\cite{77}).

We consider the above problem in the \emph{modified mild formulation} \ref{secondaccent}.
For finite element spaces $P^\delta \times U^\delta \subset H(\divv;\Omega) \times H^1(\Omega)$ and $Y_{\rm{\ref{item:b}}}^\delta \times Y_{\rm{\ref{item:c}}}^\delta \subset H_{0,\Gamma_N}(\divv;\Omega) \times H^1_{0,\Gamma_D}(\Omega)$, a Least Squares discretization using discretized dual-norms as presented in Sect.~\ref{sec:discretization}-\ref{sec:implementation} leads to the saddle-point problem to find 
$(\vec{\lambda}_{\rm{\ref{item:b}}}^\delta, \lambda_{\rm{\ref{item:c}}}^\delta,p^\delta,u^\delta) \in Y_{\rm{\ref{item:b}}}^\delta \times Y_{\rm{\ref{item:c}}}^\delta \times P^\delta \times U^\delta$ for which
\begin{align*}
&\langle \vec{\lambda}_{\rm{\ref{item:b}}}^\delta, \vec{\mu}_{\rm{\ref{item:b}}}\rangle_{H(\divv;\Omega)}+ 
\langle \nabla \lambda_{\rm{\ref{item:c}}}^\delta, \nabla \mu_{\rm{\ref{item:c}}}\rangle_{L_2(\Omega)^d}+ 
\int_{\Gamma_D} u^\delta \vec{\mu}_{\rm{\ref{item:b}}}\cdot \vec{n}\,ds+
\int_{\Gamma_N}\vec{p}^\delta\cdot\vec{n}  \mu_{\rm{\ref{item:c}}} \,ds\\
&\hspace*{4em}=
\int_{\Gamma_D} h_D \vec{\mu}_{\rm{\ref{item:b}}}\cdot \vec{n}\,ds+
\int_{\Gamma_N}h_N \mu_{\rm{\ref{item:c}}} \,ds \qquad (( \vec{\mu}_{\rm{\ref{item:b}}},\mu_{\rm{\ref{item:c}}}) \in Y_{\rm{\ref{item:b}}}^\delta \times Y_{\rm{\ref{item:c}}}^\delta),\\
&\int_{\Gamma_D} w \vec{\lambda}^\delta_{\rm{\ref{item:b}}}\cdot \vec{n}\,ds+
\int_{\Gamma_N}\vec{q}\cdot\vec{n}  \lambda^\delta_{\rm{\ref{item:c}}} \,ds
- \langle \vec{p}^\delta-\nabla u^\delta,\vec{q}-\nabla w\rangle_{L_2(\Omega)^d}
- \langle \divv \vec{p}^\delta,\divv \vec{q}\rangle_{L_2(\Omega)}\\
&\hspace*{4em}=0 \hspace*{15.5em}((\vec{q},w) \in P^\delta \times U^\delta).
\end{align*}

Let $\bbT$ denote the collection of all conforming triangulations that can be created by newest vertex bisections starting from the initial triangulation that consists of 8 triangles created by first cutting $\Omega$ along the y-axis into two equal parts, and then cutting the resulting two squares along their diagonals. The interior vertex of the initial triangulation of both squares are labelled as the `newest vertex' of all four triangles in both squares.

For $q \in \N_0$ and a family $\{\tria^\delta\}_{\delta \in \Delta} \subset \bbT$, we take
$$
P^\delta=\RT_{q}(\tria^\delta),\quad U^\delta=\cS_{q+1}^0(\tria^\delta).
$$
Now for $\delta \in \Delta$, and $\tria^\delta_D, \tria^\delta_N \in \bbT$ being the coarsest triangulations with
$\cF(\tria^\delta_D) \cap \Gamma_D =\cF(\tria^\delta) \cap \Gamma_D$ and $\cF(\tria^\delta_N) \cap \Gamma_N =\cF(\tria^\delta) \cap \Gamma_N$, we take
$$
Y_{\rm{\ref{item:b}}}^\delta:= \RT_{q+1}(\tria_D^\delta) \cap H_{0,\Gamma_N}(\divv;\Omega),\quad 
Y_{\rm{\ref{item:c}}}^\delta:=\cS^0_{d+q}(\tria^\delta_N) \cap H^1_{0,\Gamma_D}(\Omega).
$$
Then as follows from Theorems~\ref{thm: infsup1}, \ref{thm: infsup2}, and \ref{thm:quasi-opt}, $(p^\delta,u^\delta)$ is a \emph{quasi-best} approximation to $(\nabla u, u)$ from $P^\delta \times U^\delta $ w.r.t.~the norm on $H(\divv;\Omega)\times H^1(\Omega)$.
\footnote{Notice that if, for convenience of implementation, one would take $\tria^\delta_D=\tria^\delta_N=\tria^\delta$, then obviously the same result holds true.}

As follows from Remark~\ref{rem:2}, 
\begin{align*}
\mathcal{E}(\vec{\lambda}_{\rm{\ref{item:b}}}^\delta, \lambda_{\rm{\ref{item:c}}}^\delta, 
&u^\delta, p^\delta):= \\ &\sqrt{\|\vec{\lambda}_{\rm{\ref{item:b}}}^\delta\|_{H(\divv;\Omega)}^2+\|\lambda_{\rm{\ref{item:c}}}^\delta\|_{H^1(\Omega)}^2+\|\vec{p}^\delta-\nabla u^\delta\|_{L_2(\Omega)^d}^2+\|\divv p^\delta\|_{L_2(\Omega)^d}^2}
\end{align*}
is an efficient, and, by Remarks~\ref{rem:3}, \ref{rem:dataosc1}, and \ref{rem:dataosc2}, asymptotically reliable estimator for the error
$\sqrt{\|\nabla u -p^\delta\|_{H(\divv;\Omega)}^2+\|u-u^\delta\|_{H^1(\Omega)}^2}$.

Because of the limited smoothness of the solution, the asymptotic convergence rate for uniform refinements cannot be expected to exceed $\frac14$.
To drive an adaptive scheme, the error estimator needs to be split into element-wise contributions.
While it is natural to split $\|\vec{p}^\delta-\nabla u^\delta\|_{L_2(\Omega)^d}^2$ and $\|\divv p^\delta\|_{L_2(\Omega)^d}^2$
into contributions corresponding to elements $K\in\mathcal{T}^\delta$, a similar splitting of $\|\vec{\lambda}_{\rm{\ref{item:b}}}^\delta\|_{H(\divv;\Omega)}^2$ and $\|\lambda_{\rm{\ref{item:c}}}^\delta\|_{H^1(\Omega)}^2$, although possible, would require additional work since 
$\vec{\lambda}_{\rm{\ref{item:b}}}^\delta$ and $\lambda_{\rm{\ref{item:c}}}^\delta$ are finite element functions w.r.t.~partitions $\tria_D^\delta$ and $\tria_N^\delta$, which generally are coarser than $\mathcal{T}^\delta$. Since we expect, however, that $\vec{\lambda}_{\rm{\ref{item:b}}}^\delta$ and $\lambda_{\rm{\ref{item:c}}}^\delta$ have their largest values at elements at the Dirichlet or Neuman boundary, we will ignore their contributions to the estimator associated to other elements. 
Making use of the fact that $\cF(\tria^\delta_D) \cap \Gamma_D =\cF(\tria^\delta) \cap \Gamma_D$ and $\cF(\tria^\delta_N) \cap \Gamma_N =\cF(\tria^\delta) \cap \Gamma_N$, for $K \in \mathcal{T}^\delta$ we define the local estimator as follows
\begin{align*}
\eta_K^2:=\|\vec{p}^\delta&-\nabla u^\delta\|_{L_2(K)^d}^2+\|\divv p^\delta\|_{L_2(K)^d}^2\\
&+\left\{\begin{array}{@{}ll} \|\vec{\lambda}_{\rm{\ref{item:b}}}^\delta\|_{H(\divv;K')}^2 & \text{when }
\partial K \cap \Gamma_D=\partial K' \cap \Gamma_D \neq \emptyset,\text{ and } K' \in \tria_D^\delta,\\ \|\lambda_{\rm{\ref{item:c}}}^\delta\|_{H^1(K')}^2 & \text{when }\partial K \cap \Gamma_N=\partial K' \cap \Gamma_N \neq \emptyset, \text{ and } K' \in \tria_N^\delta \end{array}\right.
\end{align*}
The results given in Figure~\ref{figure:fem} indicate that the adaptive routine driven by $\{\eta_K^2\colon K \in \tria^\delta\}$ with bulk chasing parameter $\theta=0.6$ converges with the best possible rate for both $q=0$ and $q=1$.
\begin{figure}[h]
\centering
\begin{subfigure}{0.5\textwidth}
\includegraphics[width = \textwidth]{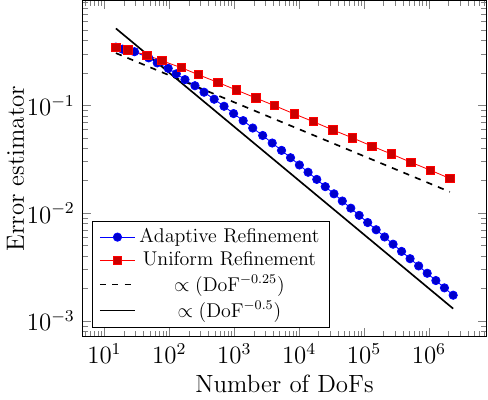}
\label{figure:fem_a}
\end{subfigure}%
\begin{subfigure}{0.5\textwidth}
\includegraphics[width = \textwidth]{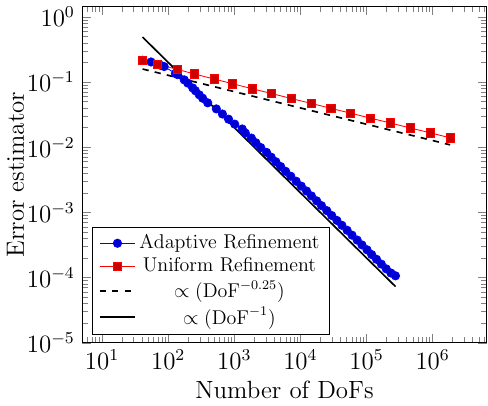}
\label{figure:fem_b}
\end{subfigure}
\begin{subfigure}{0.5\textwidth}
\includegraphics[width = \textwidth]{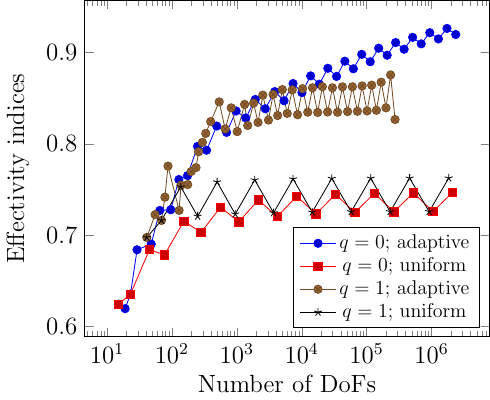}
\label{figure:fem_c}
\end{subfigure}%
\begin{subfigure}{0.5\textwidth}
\includegraphics[width = \textwidth]{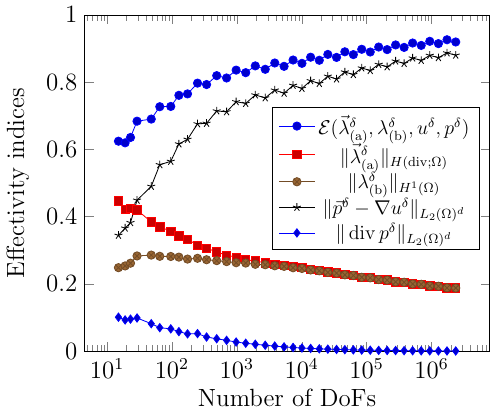}
\label{figure:fem_d}
\end{subfigure}
\caption{Results with Sect.~\ref{sec:numfem} (FEM). Top row: \#DoFs in $P^\delta\times U^\delta$ vs. $\mathcal{E}(\vec{\lambda}_{\rm{\ref{item:b}}}^\delta, \lambda_{\rm{\ref{item:c}}}^\delta, u^\delta, p^\delta)$ for $q=0$ (left) and $q=1$ (right). \newline
Bottom left: \#DoFs in $P^\delta\times U^\delta$ vs. effectivity index $\mathcal{E}(\vec{\lambda}_{\rm{\ref{item:b}}}^\delta, \lambda_{\rm{\ref{item:c}}}^\delta, u^\delta, p^\delta)/ \sqrt{\|\nabla u -p^\delta\|_{H(\divv;\Omega)}^2+\|u-u^\delta\|_{H^1(\Omega)}^2}$. \newline
Bottom right:  \#DoFs in $P^\delta\times U^\delta$ vs. different parts of the error estimator $\mathcal{E}(\vec{\lambda}_{\rm{\ref{item:b}}}^\delta, \lambda_{\rm{\ref{item:c}}}^\delta, u^\delta, p^\delta)$, all multiplied with $1/\sqrt{\|\nabla u -p^\delta\|_{H(\divv;\Omega)}^2+\|u-u^\delta\|_{H^1(\Omega)}^2}$, for adaptive refinement with $q=0$.}
\label{figure:fem}
\end{figure}
An example of the triangulation $\tria^\delta$, and corresponding triangulations $\tria_D^\delta$ and $\tria_N^\delta$ is given in Figure~\ref{fig:meshes}.
\begin{figure}[h]
 \includegraphics[angle=90,width = 3cm]{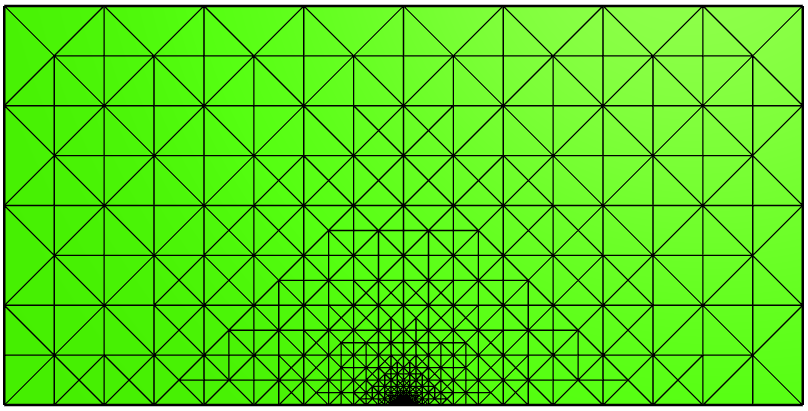}\quad
  \includegraphics[angle=90,width = 3cm]{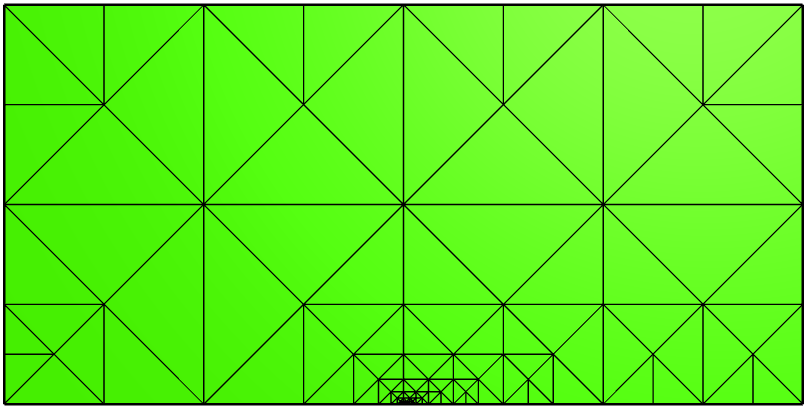}\quad
   \includegraphics[angle=90,width = 3cm]{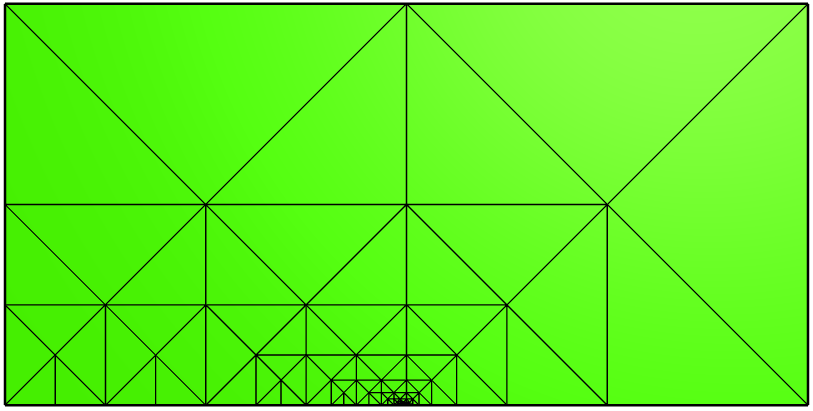}
       \caption{Results with Sect.~\ref{sec:numfem} (FEM). Meshes $\tria^\delta$, $\tria_D^\delta$ and $\tria_N^\delta$, where $\Omega$ has been rotated over $90^\circ$ counterclockwise.}
    \label{fig:meshes}
\end{figure}

\subsection{Experiments using Machine Learning}
Considering the elliptic model problem

$$ \left\{
\begin{array}{r@{}c@{}ll}
-\triangle u&\,\,=\,\,& g &\text{ on } \Omega,\\
u &\,\,=\,\,& h &\text{ on } \partial\Omega,
\end{array}
\right.
$$
we will compare our newly introduced methods to three prevalent archetypes of Machine Learning approaches: Deep Ritz Method (DRM) \cite{70.25}, Physics Informed Neural Network (PINN) \cite{243.65}, and Weak Adversarial Network (WAN) \cite{19.28}. The first of these methods (DRM) minimises 
the energy functional, the second method (PINN) minimises the 
squared $L_2(\Omega)$-norm of the residual, while the latter (approximately) minimises the squared $H^{-1}(\Omega)$-norm of the residual. All three however deal with the essential boundary condition through the addition of a multiple of the squared $L_2(\partial\Omega)$-norm of the boundary residual to the functional to be minimised .

In general, all these methods minimise a loss functional $L(\bigcdot)$ over some (deep) neural network $\cX \subset X$, finding an approximate solution $u_\cX := \argmin_{w \in \cX} L(w)$. In the aforementioned examples, $X=H^1(\Omega)$ and the loss functions are as follows

\begin{subequations}
\begin{align*}
    \text{(DRM)} \quad L(w) &= \tfrac{1}{2}\|\nabla w\|_{L_2(\Omega)}^2 - g(w) + \alpha \|w-h\|_{L_2(\partial \Omega)}^2 \;,\\
    \text{(PINN)} \quad L(w) &= \|g + \Delta w \|_{L_2(\Omega)}^2 + \alpha \|w-h\|_{L_2(\partial \Omega)}^2 \;, \\
    \text{(WAN)} \quad  L(w) &= \sup_{v \in \cY} \frac{\left(\int_\Omega  \nabla w \cdot \nabla(\phi v) \, \mathrm{d}x -g(\phi v)\right)^2}{\|\nabla(\phi v)\|_{L_2(\Omega)^d}^2} +\alpha \|w-h\|_{L_2(\partial\Omega)}^2 \;,
\end{align*}
\end{subequations}
where $\phi \in H^1_0(\Omega)$ with $\phi(x) \eqsim \text{dist}(x, \partial \Omega)$, $\alpha > 0$ some chosen constant, and $\cY \subset Y \coloneqq H^1(\Omega)$ a neural network like $\cX$. Note that in the WAN case, due to the need of evaluating a dual-norm, there is an additional supremum of $v \in \cY$. For reasons explained in Sect.~\ref{sec:6.1}, it is best to rewrite the WAN loss function using Lemma \ref{lem:1} in the form

\begin{subequations}
    \begin{align*}
        \text{(WAN)} \quad L(w) = \sup_{v \in \cY}\left\{\int_\Omega \nabla w \cdot\nabla (\phi v)\,\mathrm{d}x - g(\phi v)- \tfrac12 \|\nabla(\phi v)\|_{L_2(\Omega)^d}^2\right\} \\+\alpha \|w-h\|_{L_2(\partial\Omega)}^2.
    \end{align*}
\end{subequations}

We compare the above methods with the following four newly introduced least square loss functions whose minimum over the neural network $\cX$ will produce a quasi-best solution from $\cX$, given that the neural network $\cY$ at the test side is big enough:


\begin{subequations}
$\textrm{\normalfont{(QOLS1)}} \quad \vec{w} \in \cX \subset X \coloneqq H^1(\Omega) \times H(\divv;\Omega), \quad  \cY \subset Y \coloneqq H(\divv;\Omega),$
\begin{align*}
L(\vec{w}) \coloneqq L(w, \vec{q}) &= \tfrac12
\|\vec{q} -\nabla w\|_{L_2(\Omega)^d}^2+\tfrac12 \|\divv \vec{q}+g\|_{L_2(\Omega)}^2
\\
&\hspace*{1em}+\sup_{\vec{v} \in \cY} \left\{\int_{\partial\Omega}(w-h) \gamma_{\partial\Omega}^{\vec{n}}(\vec{v})\,\mathrm{d}s -\tfrac12\|\vec{v}\|_{H(\divv; \Omega)}^2\right\},
\end{align*}
$\textrm{\normalfont{(QOLS1$_\Delta$)}} \quad \vec{w} \in \cX \subset X \coloneqq  H^1(\Omega) \times H(\divv;\Omega), \quad  \cY \subset Y \coloneqq H_\triangle(\Omega),$
\begin{align*}
L(\vec{w}) \coloneqq L(w, \vec{q}) &= \tfrac12
\|\vec{q} -\nabla w\|_{L_2(\Omega)^d}^2+\tfrac12 \|\divv \vec{q}+g\|_{L_2(\Omega)}^2
\\
&\hspace*{1em}+\sup_{v \in \cY} \left\{\int_{\partial\Omega}(w-h) \gamma_{\partial\Omega}^{\vec{n}}(\nabla v)\,\mathrm{d}s -\tfrac12\|v\|_{H_\triangle(\Omega)}^2\right\},
\end{align*}
$\textrm{\normalfont{(QOLS2)}} \quad w \in \cX \subset X \coloneqq H^1(\Omega), \quad \cY \subset Y \coloneqq H^1(\Omega) \times H(\divv;\Omega),$
\begin{align*}
L(w) &= \sup_{\vec{v}=(v_1,\vec{v_2}) \in \cY}
\int_\Omega  \nabla w \cdot\nabla \phi v_1\,\mathrm{d}x-g(\phi  v_1)\\
&\hspace*{1em} +\int_{\partial\Omega} (w-h) \gamma_{\partial\Omega}^{\vec{n}} ( \vec{v_2})\,\mathrm{d}s -\tfrac12 \|\nabla(\phi  v_1)\|_{L_2(\Omega)^d}^2 -\tfrac12\|\vec{v_2}\|_{H(\divv;\Omega)}^2,
\end{align*}
$\textrm{\normalfont{(QOLS2$_\Delta$)}} \quad w \in \cX \subset X \coloneqq H^1(\Omega), \quad \cY \subset Y \coloneqq H^1(\Omega) \times H_{\triangle}(\Omega),$
\begin{align*}
L(w) &= \sup_{\vec{v}=(v_1,v_2) \in \cY}
\int_\Omega  \nabla w \cdot\nabla \phi v_1\,\mathrm{d}x-g(\phi  v_1)\\
&\hspace*{1em} +\int_{\partial\Omega} (w-h) \gamma_{\partial\Omega}^{\vec{n}} (\nabla v_2)\,\mathrm{d}s -\tfrac12 \|\nabla(\phi  v_1)\|_{L_2(\Omega)^d}^2 -\tfrac12\|v_2\|_{H_{\triangle}(\Omega)}^2.
    \end{align*}
\end{subequations}

$\textrm{\normalfont{(QOLS1$_\Delta$)}}$ and $\textrm{\normalfont{(QOLS2$_\Delta$)}}$ are the first order system and second order least squares formulations from \eqref{eq:103} and \eqref{eq:104}, specialized to the case that $A=I$, $B=0$ and $\gamma_D=\partial\Omega$, and $\textrm{\normalfont{(QOLS1)}}$ and $\textrm{\normalfont{(QOLS2)}}$ are the corresponding formulations where the Dirichlet boundary condition is enforced by means of \eqref{eq:100} instead of \eqref{eq:102}.

As with the WAN method, these formulations involve solving a minimax problem. To solve this in practise, one therefore needs to switch between minimising the loss function over the test space $\cX$ for $K_\text{w}$ steps, and maximising over the trial space $\cY$ for $K_\text{v}$ steps. One might consider a more intelligent switching between minimizing and maximizing, but this is beyond the scope of this paper. \medskip

As the given integrals so far will most often not have a closed form, these must be approximated. This can either be done with Monte Carlo integration or quadrature integration, each having its pros and cons. \medskip

Monte Carlo integration does not suffer from the curse of dimensionality while also inducing a stochastic property to our integral, which counteracts overfitting of the model on quadrature points and escaping local minima in gradient descent. Quadrature integration on the other hand converges faster and gives more control over the accuracy of the approximation, while overfitting can be counteracted through adaptive quadrature strategies \cite{RIVERA2022114710}. \\

For our Monte Carlo integration, we will use simple uniform sampling of the domain, which results in the following approximations

\begin{subequations}
    \begin{align*}
        \int_{\Omega} g(x)\,\mathrm{d}x \approx \frac{|\Omega|}{N_{\text{r}}} \sum_{i=1}^{N_{\text{r}}} g(x_i^{\text{r}}), \quad \quad \int_{\partial \Omega} h(x)\,\mathrm{d}s \approx \frac{|\partial \Omega|}{N_{\text{b}}} \sum_{i=1}^{N_{\text{b}}} h(x_i^{\text{b}}),
\end{align*}
\end{subequations}
where the $\{x_i^{\text{r}}\}_{i=1}^{N_{\text{r}}} \subset \Omega$ and $\{x_i^{\text{b}}\}_{i=1}^{N_{\text{b}}} \subset \partial \Omega$ are sampled uniformly from their respective domains. It is possible to adopt different sampling strategies to decrease the variance (such as importance sampling \cite{https://doi.org/10.1111/mice.12685} or quasi-Monte Carlo strategies \cite{chen2019quasimonte}), but this is not the focus of this paper and we will therefore stick to the straightforward uniform sampling. \medskip

For the numerical experiments on a  domain $\Omega$ that is the union of $d$-dimensional hypercubes, an adaptive tensor product Gauss-Legendre quadrature scheme was implemented. This was done by initially subdividing the domain $\Omega$ into $N$ hypercubes $\Omega_i$. The scheme was then made adaptive by dividing each $\Omega_i$ into $2^d$ hypercubes $\Omega_{i_k}$, and computing the following criterium
\begin{equation*}
   \left| \mathrm{GL}(f;\Omega_i) - \sum_{k=1}^{2^d} \mathrm{GL}(f;\Omega_{i_k}) \right| < \tau_1 \frac{\text{diam}(\Omega_i)}{|\Omega|} \max\left\{\left|\sum_{i}\mathrm{GL}(f;\Omega_i) \right|, \tau_2 \right\},
\end{equation*}
where with $\mathrm{GL}(f, \omega)$ we denote the Gauss-Legendre tensor product quadrature of a function $f$ over a hypercube $\omega$. If the criterium holds true, the subdomain $\Omega_i$ is accepted and does not to be refined further, otherwise it is rejected and gets replaced by its refined subdomains. Ideally the refinements stops when all subdomains are accepted, in which case one expects either an absolute error of the order $\tau_1 \tau_2$ or a relative error of the order $\tau_2$, where $\tau_1$ and $\tau_2$ are some chosen constants. Unfortunately this can be expensive (i.e.~in the case of a singular function $f$), so for practical reasons the refinements will automatically be stopped after 1000 rejections and consequent refinements of the domain. 
\medskip

This then leaves us with defining the deep neural networks $\cX \subset X$ and $\cY \subset Y$ in which our trial and test solution will lie. For the architecture of these networks, we have opted for the so-called Residual Neural Networks (ResNet). These were introduced in \cite{He2015DeepRL}, and were designed to avoid the vanishing gradient problem, making them easier to train. A graphical representation of the ResNet we used, is given in Figure~\ref{fig:ResNet}. This network is built out of two main components: the ResNet blocks and linear transformations, one to get the input vector to same width as the ResNet block and one to get the output of the last ResNet block to the correct output dimension. \\

\begin{figure}
  \includegraphics[width=\textwidth]{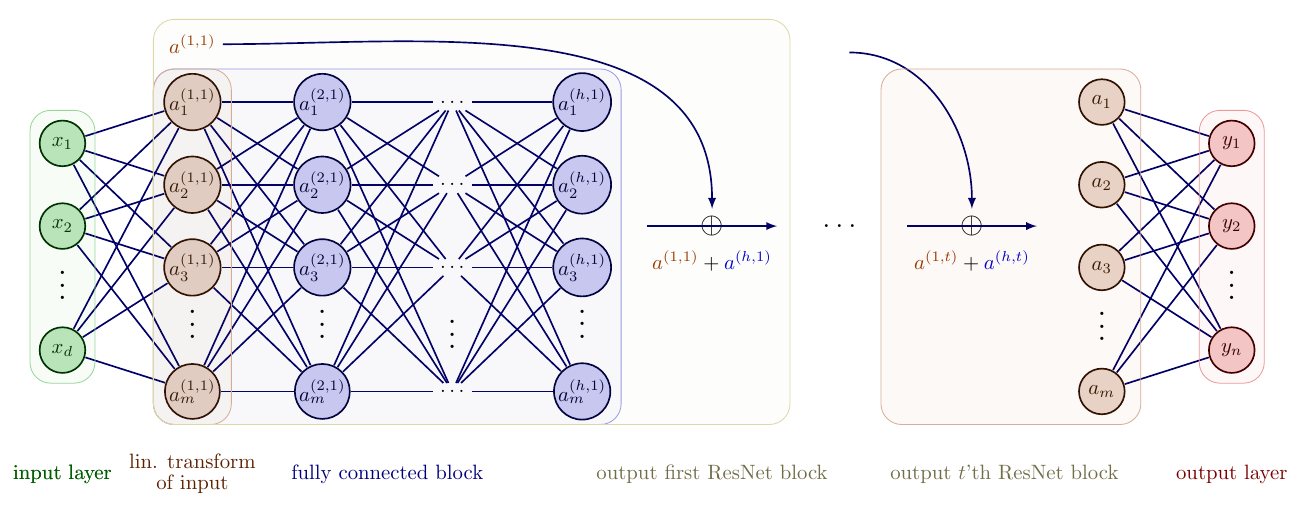}
  \caption{Graphical representation Residual Neural Network (ResNet).}
  \label{fig:ResNet}
\end{figure}

Using the same notation as in Figure~\ref{fig:ResNet}, the network first calculates 
\begin{equation*}
    a^{(1,1)} = \mathcal{F}_{0}(x) \coloneqq  W_0 x + b_0,
\end{equation*}
in which $W_0 \in \mathbb{R}^{m \times d}$ and $b_0 \in \mathbb{R}^{m}$. Note that the notation $a^{(j,k)}$ signifies the $j$'th layer of the $k$'th ResNet block. With $a^{(1,1)}$, we have entered the first ResNet block. This is a fully connected neural network (FCNN) of width $m$ and depth $h$. This means that to get from layer $a^{(j,k)}$ to $a^{(j+1,k)}$ for $1 \leq k \leq t$, one calculates

\begin{equation*}
    a^{(j+1,k)} = \mathcal{F}_{j,k}(a^{(j,k)}) \coloneqq \sigma\left(W_{j,k}a^{(j,k)} + b_{j,k} \right), \qquad \text{for } 1 \leq j \leq h -1,
\end{equation*}
where $W_{j,k} \in \mathbb{R}^{m\times m}$, $b_{j,k} \in \mathbb{R}^m$ and $\sigma$ is some non-linear activation function which acts coordinate wise. 
To ensure that this architecture does not suffer from a vanishing gradient, 
the final output of the $k$'th ResNet block is given by the addition of the first layer of its FCNN with the last

\begin{equation*}
    a^{(1,k+1)} = \mathcal{F}_{h,k}(a^{(1,k)},a^{(h,k)}) \coloneqq a^{(1,k)} + a^{(h,k)}, \qquad \text{for } 1 \leq k \leq t.
\end{equation*}
We can therefore define 

\begin{equation*}
    \mathcal{F}_{k}(a) \coloneqq \mathcal{F}_{h,k} \circ \mathcal{F}_{h-1,k}\circ \cdots \circ \mathcal{F}_{1,k}(a) \qquad \text{for } 1 \leq k \leq t,
\end{equation*}
which maps from the initial layer of the $k$'th ResNet block, to its output. The last step is a linear transformation from the final ResNet block output $a \coloneqq a^{(1,t+1)}$ to the output vector $y$, which is given by

\begin{equation*}
    y = \mathcal{F}_{t+1}(a) \coloneqq W_{t+1}a + b_{t+1},
\end{equation*}
where $W_{t+1} \in \mathbb{R}^{n \times m}$ and $b_{t+1} \in \mathbb{R}^{n}$. We therefore find that 
\begin{equation}\label{eq: forward Resnet}
    y = \mathcal{F}(x) \coloneqq \mathcal{F}_{t+1} \circ \mathcal{F}_{t} \circ \cdots \circ \mathcal{F}_{0}(x).
\end{equation}
In the context of neural networks, this function $\mathcal{F}$ is often called the `forward' function of the network. 

The set of all possible different parameters of our ResNet will be denoted with
\begin{equation*}
\begin{split}
    \Theta = \big\{&\big(W_0, W_{t+1}, b_0, b_{t+1}, W_{j,k}, b_{j,k} \big) \colon W_0 \in \mathbb{R}^{m \times d}, b_0 \in \mathbb{R}^m, W_{t+1}\in \mathbb{R}^{n \times m},\\ &b_{t+1}\in \mathbb{R}^n, W_{j,k} \in \mathbb{R}^{m \times m}, b_{j,k} \in \mathbb{R}^m, \text{ with } 1 \leq j \leq h-1,\, 1 \leq k \leq t \big\}.
\end{split}
\end{equation*}
This allows us to define our neural network space of functions as 
\begin{equation*}
    \text{ResNet}(d,n,h,m,t) \coloneqq \left\{u_{\theta} \colon \mathbb{R}^d \to \mathbb{R}^n \colon \theta \in \Theta \right\},
\end{equation*}
where $u_\theta$ denotes the forward function $\mathcal{F}$ as defined in \eqref{eq: forward Resnet} for a specific choice of parameters $\theta \in \Theta$. Note that the number of degrees of freedom for our ResNet is given by $\dim(\Theta) = m(1+d+n+t(h-1)(m+1)) + n$. \medskip

Having now defined all necessary components, we can formulate the general algorithm for solving our model problem

\begin{algorithm}[hbt!]
\caption{Quasi-Optimal Least Squares (QOLS) method}\label{alg:QOLS}
\begin{algorithmic}
\State{\textbf{Input:} Loss function $L(w_\theta,v_\eta)$: One of the four given QOLS formulations.\\
The chosen integration scheme: Monte Carlo integration or adaptive quadrature rule integration.\\
$K_{\text{w}}/K_{\text{v}}$: The number of updates of the trial/test solutions per epoch. \\
$N_{\text{ep}}$: The total number of epochs}
\State{\textbf{Initialize:} The trial solution $w_\theta \in \cX$ and test solution $v_\eta \in \cY$};
\For{epoch $= 1,\ldots, N_{\text{ep}}$}
 \For{step $= 1,\ldots, K_\text{w}$}
   \State Compute the loss $L(w_\theta,v_\eta)$;
   \State Update trial function parameters $\theta$ by minimising loss through AdamW;
  \EndFor
  \For{step $= 1,\ldots, K_\text{v}$}
   \State Compute the loss $L(w_\theta,v_\eta)$;
   \State Update test function parameters $\eta$ by maximising loss through AdamW;
  \EndFor
\EndFor
\State{\textbf{Output:} Solution $w_\theta$};
\end{algorithmic}
\end{algorithm}

This algorithm was implemented for each of the different QOLS formulations along with the DRM, WAN and PINN in Python using PyTorch.\medskip

\begin{example} \label{example:1}
To compare these methods we will look at the following problem

$$ \left\{
\begin{array}{@{}r@{}c@{}ll}
-\triangle u(x) &\,\,=\,\,& 0, &x\in \Omega,\\
u(x) &\,\,=\,\,& u(r,\theta) = r^{\frac{2}{3}}\sin \frac{2}{3}\theta, &x\in\partial\Omega,
\end{array}
\right.
$$
with the L-shaped domain $\Omega = (-1,1)^2 \setminus (0,1) \times (0,-1)$.
Note that the exact solution $u^\ast(x) = u^\ast(r,\theta)= r^{\frac{2}{3}}\sin \frac{2}{3}\theta$
 is only in $H^s(\Omega)$ for $s<\frac53$, with its gradient blowing up at the origin. Because of the non-smoothness of $u$ at $\partial\Omega$ we expect that our correct imposition of the boundary condition will show-off.

The networks for $\cX$ and $\cY$ in each of the different methods are chosen to be $\text{ResNet}(d=2,n,h=2,m=30,t=4)$, where $n$ is either $1,2$ or $3$, depending on the required amount of outputs, which yields $3841$, $3872$ and $3903$ degrees of freedom respectively. The activation function between all layers is chosen to be the Exponential Linear Unit (ELU), as it is a smooth function guaranteeing that $\cX \subset X$ and $\cY \subset Y$. \medskip

For the methods (WAN, QOLS2, QOLS2$_\Delta$) that require a $\phi(x) \eqsim \text{dist}(x, \partial \Omega)$, one could define it to be the exact distance to the boundary but this will result in a $\phi$ with $\nabla \phi$ being discontinuous, which can lead to numerical integration issues. Instead we will define $\phi$ as 
$$
\phi(x)=\left\{\begin{array}{cc} \Big(\frac{1}{\sum_i (1/a_i(x))^p}\Big)^{1/p} & \text{ when } \min_i a_i(x) >0\\
0 & \text{ otherwise}
\end{array} \right.
$$
for $p=2$, where the $a_i(x)$ are the exact distance functions to each of the sides of our domain $\Omega$. With this definition, one finds that $\phi \in C^\infty(\Omega)$ and $\phi(x) \eqsim \text{dist}(x, \partial \Omega)$ \cite{255.5}. \medskip

For backward propagation of the models ($u_\theta$ and $v_\eta$), we will use the popular AdamW algorithm for gradient descent, with the learning rate set to the recommended $\text{lr}=0.001$. For the QOLS methods, we let the learning rate drop off slowly by multiplying it with $0.99$ every $100$ epochs (resulting in a learning rate of approximately $\text{lr}=0.00022$ at the end). We do this as we found experimentally that these methods already reach a local minimum very rapidly for $\text{lr}=0.001$, and decreasing the learning rate somewhat helps stabilize them. \medskip

For the numerical integration, in the case of Monte Carlo integration we set $N_r = 4000$ and $N_b = 1000$, while for the adaptive quadrature we used Gauss-Legendre integration with polynomial degree $3$, an initial subdivision of $\Omega$ into $N = 192$ subdomains, a maximum of $1000$ further refinements and tolerances $\tau_1 = \tau_2 = 0.001$. \medskip

Lastly we ran all methods for $N_{\text{ep}} = 15000$ epochs, chose $K_{\text{w}} = 1,$ $K_{\text{v}} = 10$ and, for the existing machine learning approaches, set the parameter $\alpha = 500$. This choice for $\alpha$ seems to be a pretty optimal choice from experiments. The results using adaptive quadrature integration are presented in Figure \ref{fig:ML results quad adap} and are very similar to the results using Monte Carlo integration, the main difference being that the former managed an extra decrease in $H^1$-error of about 2-4 time for the smallest $H^1$-errors.

\begin{figure}[h!]
  \includegraphics[width=0.8\textwidth]{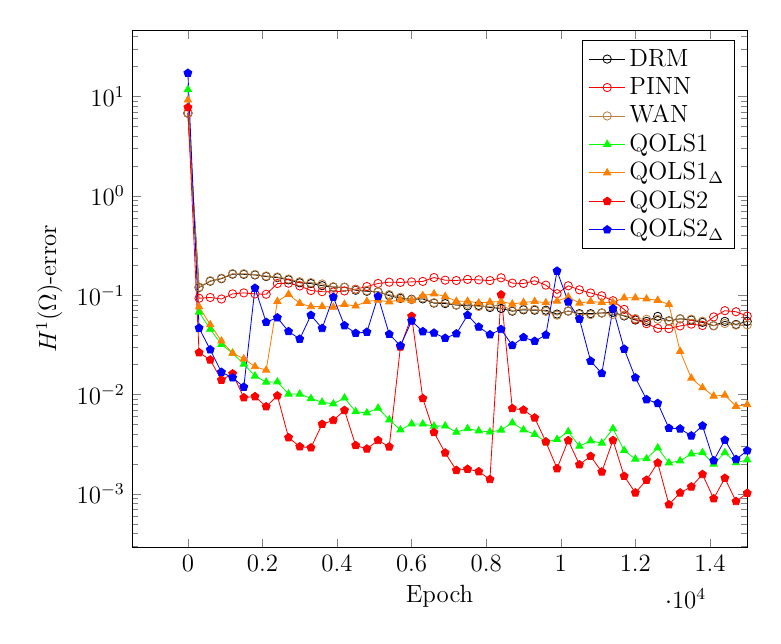}
  \caption{Illustration with Ex.~\ref{example:1} (Machine Learning). Plot of the $H^1(\Omega)$-error $\|u^\ast - u_\theta\|^2_{H^1(\Omega)}$ against epoch for the different methods, using adaptive quadrature integration.}
  \label{fig:ML results quad adap}
\end{figure}

From these results we find that for this problem our newly introduced QOLS formulations outperform the DRM, PINN and WAN methods. It could be that given enough epochs, these methods catch up, but at least on this time scale for similar learning rates the QOLS formulations perform (up to 100x) better. 
For smooth problems (e.g. $u^\ast(x) = \|x\|_2^2$) the performance gap as in Figure \ref{fig:ML results quad adap} is much less prevalent if present at all. \medskip


Comparing the QOLS methods among themselves, we find the first order system formulations slightly outperforming the second order formulations. This does come at the price of 
needing vector-valued functions within the trial space $\cX$ possibly necessitating more network parameters for higher dimensional problems to model these extra outputs. Furthermore we find that the $\Delta$ formulations slightly underperform compared to their counterparts, but one can expect for the aforementioned reasons that for higher dimensional problems this will turn around as the $\Delta$ formulations have scalar-valued functions in their test space $\cY$.\medskip

On a practical note, the QOLS formulations have the advantage of not having to bother with making a good choice for $\alpha$, a choice which greatly affects the performance of the DRM, PINN and WAN methods that can only be attained empirically. A downside however is that the QOLS formulations, similar to WAN, have to deal with an adversarial network and solve a min-max problem as a result of minimizing a dual norm. This is a lot more expensive and as far as we know, there is no intelligent way of choosing when to switch between the inner and outer loop.
\end{example}

\section{Conclusion} \label{sec:conclusion}
We have introduced least-squares formulations of second order elliptic equations and the stationary Stokes equations with possibly inhomogeneous boundary conditions, whose minimization over a finite element space or a Neural Net yield approximations that are quasi-best. This is due to the fact that the least-squares residual is equivalent to the (squared) error in a canonical norm. 
Despite of that, the use of fractional Sobolev norms on the boundary, or even any function space on the boundary has been avoided. 
Such spaces were replaced by the ranges of trace operators of standard function spaces on the domain. Some parts of the residual are measured in dual norms. For sufficiently large test finite element spaces or adversarial Neural Nets, they are replaced by discretized dual norms whilst preserving quasi-optimal approximations. Numerical results both for finite element spaces and for Neural Nets are presented. The advantage of our approach compared to usual Machine Learning algorithms is apparent for solutions that have singularities. It is however fair to say that known problems with Machine Learning algorithms for solving PDEs, as quadrature issues or the painful minimization of a non-convex functional, are not solved by the use of our well-posed least-squares functionals.

\appendix 
\section{Proof of Proposition~\ref{prop:stokes}} \label{sec:app}
By making the substitutions
\be \label{eq:11}
\undertilde{\vec{u}}:=\vec{u} \sqrt{\nu},\,
\undertilde{p}:=p/\sqrt{\nu},\,
\undertilde{\vec{f}}:=\vec{f}/\sqrt{\nu},
\undertilde{g}:=g\sqrt{\nu},\,
\undertilde{\vec{h}}:=\vec{h}\sqrt{\nu},\,
\ee
the modified variables satisfy the Stokes system \eqref{eq:Stokesequations} with $\nu=1$.

By testing the first and second equation of this system with $\vec{v} \in H_0^1(\Omega)^d$ and $q \in L_{2,0}(\Omega)$, respectively, and applying integration by parts we obtain
$$
\int_\Omega \nabla \undertilde{\vec{u}}: \nabla \vec{v}-\undertilde{p}\divv \vec{v}+q \divv \undertilde{\vec{u}}\,dx=\int_{\Omega}\undertilde{\vec{f}}\cdot\vec{v}+ \undertilde{g} q\,dx.
$$
From $\triangle=\nabla \divv+\curl \curl $, and so $\int_\Omega \nabla \cdot : \nabla \cdot \,dx= \int_\Omega \divv \cdot  \divv \cdot +\curl \cdot  \curl \cdot \,dx
$ on $H^1(\Omega)^d \times H_0^1(\Omega)^d$, regardless of the boundary datum $\undertilde{\vec{h}}$ the Stokes solution $(\undertilde{\vec{u}},\undertilde{p})$ satisfies
\be \label{eq:8}
\int_\Omega \divv \undertilde{\vec{u}}\divv \vec{v}+\curl \undertilde{\vec{u}}\cdot \curl \vec{v}-\undertilde{p}\divv \vec{v}+q \divv \undertilde{\vec{u}}\,dx=\int_{\Omega}\undertilde{\vec{f}}\cdot\vec{v}+\undertilde{g} q\,dx
\ee
($\vec{v} \in H_0^1(\Omega)^d,\,q \in L_{2,0}(\Omega)$). 
The operator defined by the left hand side is in $\cL(H^1(\Omega)^d\times L_{2,0}(\Omega),H^{-1}(\Omega)^d\times L_{2,0}(\Omega))$. Furthermore, it is well-known that this operator is in $\Lis(H_0^1(\Omega)^d\times L_{2,0}(\Omega),H^{-1}(\Omega)^d\times L_{2,0}(\Omega))$.

To arrive at a first order system, we introduce the vorticity $\undertilde{\vec{\omega}}:=\curl \undertilde{\vec{u}}$ as an additional variable. Then \eqref{eq:8} is equivalent to the system
\be \label{eq:9}
 \left\{
\begin{array}{r@{}c@{}ll}
\int_\Omega \divv \undertilde{\vec{u}}\divv \vec{v}+\undertilde{\vec{\omega}} \cdot \curl \vec{v}-\undertilde{p}\divv \vec{v}+q \divv \undertilde{\vec{u}}\,dx &\,\,=\,\,&\int_{\Omega}\undertilde{\vec{f}}\cdot\vec{v}+\undertilde{g} q\,dx\\
\undertilde{\vec{\omega}}-\curl \undertilde{\vec{u}}&\,\,=\,\,&0
\end{array}
\right.
\ee
($\vec{v} \in H_0^1(\Omega)^d,\,q \in L_{2,0}(\Omega)$). 
With $U:=H^1(\Omega)^d\times L_{2,0}(\Omega) \times L_2(\Omega)^{2d-3}$, $U_0:=H_0^1(\Omega)^d\times L_{2,0}(\Omega) \times L_2(\Omega)^{2d-3}$, and $V:=H^{-1}(\Omega)^d \times L_{2,0}(\Omega) \times L_2(\Omega)^{2d-3}$, the operator defined by the left hand side, which we will denote by $H$, satisfies $H \in \cL(U,V)$. 
To see that $H \in \Lis(U_0,V)$, replace the zero at the right hand side of the second equation by $\vec{k} \in L_2(\Omega)^{2d-3}$. By substituting $\undertilde{\vec{\omega}}=\vec{k}+\curl \undertilde{\vec{u}}$ in the first equation, by the well-posedness of \eqref{eq:8} one finds a solution $(\undertilde{\vec{u}},\undertilde{p},\undertilde{\vec{\omega}}) \in 
U_0$ whose norm is bounded by the norm of the data in $V$.

A solution $(\undertilde{\vec{u}},\undertilde{p},\undertilde{\vec{\omega}})$ of \eqref{eq:9} satisfies $\divv \undertilde{\vec{u}}-\undertilde{g} \in \Span \1$. Since $\vec{v} \in H_0^1(\Omega)^d$ satisfies $\divv \vec{v} \perp_{L_2(\Omega)} \1$, we conclude that
\be \label{eq:10}
 \left\{
\begin{array}{r@{}c@{}ll}
\int_\Omega \undertilde{\vec{\omega}} \cdot \curl \vec{v}-\undertilde{p}\divv \vec{v}+q \divv \undertilde{\vec{u}}\,dx &\,\,=\,\,&\int_{\Omega}\undertilde{\vec{f}}\cdot\vec{v} +\undertilde{g} (q-\divv \vec{v})\,dx\\
\undertilde{\vec{\omega}}-\curl \undertilde{\vec{u}}&\,\,=\,\,&0
\end{array}
\right.
\ee
($\vec{v} \in H_0^1(\Omega)^d,\,q \in L_{2,0}(\Omega)$). 
The operator defined by the left hand side, which we will denote by $K$, satisfies $K \in \cL(U,V)$. 
The squared norm $\|H(\undertilde{\vec{u}},\undertilde{p},\undertilde{\vec{\omega}})\|_V^2$ reads as
$$
\|\vec{v}\mapsto \int_\Omega \divv \undertilde{\vec{u}}\divv \vec{v}+\undertilde{\vec{\omega}} \cdot \curl \vec{v}-\undertilde{p}\divv \vec{v}\|^2_{H^{-1}(\Omega)^d}+\|\divv \undertilde{\vec{u}}\|^2_{L_2(\Omega)^d}+\|\undertilde{\vec{\omega}}-\curl \undertilde{\vec{u}}\|^2_{L_2(\Omega)^{2d-3}},
$$
and $\|K(\undertilde{\vec{u}},\undertilde{p},\undertilde{\vec{\omega}})\|_V^2$ reads as
$$
\|\vec{v}\mapsto \int_\Omega \undertilde{\vec{\omega}} \cdot \curl \vec{v}-\undertilde{p}\divv \vec{v}\|^2_{H^{-1}(\Omega)^d}+\|\divv \undertilde{\vec{u}}\|^2_{L_{2,0}(\Omega)^d}+\|\undertilde{\vec{\omega}}-\curl \undertilde{\vec{u}}\|^2_{L_2(\Omega)^{2d-3}}.
$$
From $\|\vec{v}\mapsto \int_\Omega \divv \undertilde{\vec{u}}\divv \vec{v}\|^2_{H^{-1}(\Omega)^d} \lesssim \|\divv \undertilde{\vec{u}}\|^2_{L_{2,0}(\Omega)^d}$, applications of the triangle inequality show $\|H(\undertilde{\vec{u}},\undertilde{p},\undertilde{\vec{\omega}})\|_V \eqsim \|K(\undertilde{\vec{u}},\undertilde{p},\undertilde{\vec{\omega}})\|_V$.
From 
$H \in \Lis(U_0,V)$, it follows that for $(\undertilde{\vec{u}},\undertilde{p},\undertilde{\vec{\omega}}) \in U_0$, $\|(\undertilde{\vec{u}},\undertilde{p},\undertilde{\vec{\omega}})\|_U \eqsim \|H(\undertilde{\vec{u}},\undertilde{p},\undertilde{\vec{\omega}})\|_V \eqsim \|K(\undertilde{\vec{u}},\undertilde{p},\undertilde{\vec{\omega}})\|_V$, i.e., $K \in \cL(U_0,V)$ is a homeomorphism with its range.
For arbitrary $(\vec{a},b,\vec{c})\in V$, there exists a $(\undertilde{\vec{u}},\undertilde{p},\undertilde{\vec{\omega}}) \in U_0$, with
$H(\undertilde{\vec{u}},\undertilde{p},\undertilde{\vec{\omega}})=(\vec{a},b,\vec{c})+(\vec{v}\mapsto\int_\Omega b \divv \vec{v}\,dx,0,0)\in V$, and so
$K(\undertilde{\vec{u}},\undertilde{p},\undertilde{\vec{\omega}})=(\vec{a},b,\vec{c})$, meaning that $K \in \Lis(U_0,V)$.

Finally, as we have seen, a consistent formulation of Stokes problem with $\nu=1$ for $(\undertilde{\vec{u}},\undertilde{p})$, with $\undertilde{\vec{\omega}}=\curl \undertilde{u}$, is given by 
$$
\undertilde{G}(\undertilde{\vec{u}},\undertilde{p},\undertilde{\vec{\omega}}):=\big(K(\undertilde{\vec{u}},\undertilde{p},\undertilde{\vec{\omega}}),\gamma(\undertilde{\vec{u}})\big)=
\big(\vec{v}\mapsto\int_\Omega \undertilde{\vec{f}} \cdot \vec{v} -\undertilde{g}\divv \vec{v}\,dx, \undertilde{g},0, \undertilde{\vec{h}}\big).
$$
It holds that $K \in \cL(U,V)$, $(\undertilde{\vec{u}},\undertilde{p},\undertilde{\vec{\omega}}) \mapsto \gamma(\vec{u}) \in \cL(U,H^{\frac12}(\partial\Omega)^d)$ is surjective, with kernel equal to $U_0$, and $K \in \Lis(U_0,V)$. By an application of Lemma~\ref{lem:bi} we conclude that  $\undertilde{G} \in \Lis(U,V \times H^{\frac12}(\partial\Omega)^d)$.

Upon substituting \eqref{eq:11} and $\omega:=\frac{\undertilde{\omega}}{\sqrt{\nu}}$, the proof of Proposition~\ref{prop:stokes} is completed.

\bibliographystyle{alpha}
\bibliography{main.bib}

\end{document}